\documentclass[a4paper,12pt]{amsart}
\usepackage{amscd,amsfonts,amssymb}
\usepackage{tikz}
\usepackage{pgfplots}
\usetikzlibrary{automata,positioning,calc,trees}
\usetikzlibrary{intersections,pgfplots.fillbetween}
\usepackage{graphicx,tikz}
\usepackage{pgf,tikz,pgfplots}
\usepackage{mathrsfs}
\usetikzlibrary{arrows}
\usepackage{enumerate}
\usepackage[shortlabels]{enumitem}
\usepackage{mathrsfs}
\usepackage{amssymb,amsmath,amsthm,color}
\usepackage{caption,subcaption}
\usepackage{hyperref}
\usepackage{url}
\usepackage{setspace}
\usepackage{float}
\date{\today}

\allowdisplaybreaks

\newtheorem{theorem}{Theorem}[section]
\newtheorem{proposition}[theorem]{Proposition}
\newtheorem{lemma}[theorem]{Lemma}
\newtheorem{definition}[theorem]{Definition}

\newtheorem{remark}[theorem]{Remark}

\def\be#1 {\begin{equation} \label{#1}}
\newcommand{\ee}{\end{equation}}

\def\sqw{\hbox{\rlap{\leavevmode\raise.3ex\hbox{$\sqcap$}}$%
\sqcup$}}
\def\findem{\ifmmode\sqw\else{\ifhmode\unskip\fi\nobreak\hfil
\penalty50\hskip1em\null\nobreak\hfil\sqw
\parfillskip=0pt\finalhyphendemerits=0\endgraf}\fi}

\newcommand{\R}{{\mathbb {R}}}
\newcommand{\Q}{{\mathbb Q}}
\newcommand{\N}{{\mathbb N}}
\newcommand{\Z}{{\mathbb Z}}
\newcommand{\C}{{\mathbb C}}

\newcommand{\K}{\mathcal K}

\newcommand{\g}{\mathfrak g}
\newcommand{\M}{\mathcal M}
\newcommand{\G}{\mathcal G}
\newcommand{\B}{\mathcal B}

\newcommand{\supp}{\operatorname{supp}}
 \author{Surjeet Singh Choudhary, Saurabh Shrivastava}
\address [Surjeet Singh Choudhary, Saurabh Shrivastava]{
Department of Mathematics\\
	Indian Institute Science Education and Research Bhopal\\
	Bhopal-462066, India}
\email{surjeet19@iiserb.ac.in, saurabhk@iiserb.ac.in}
\begin{document}
\title[On the bilinear Bochner-Riesz problem at critical index]{On the bilinear Bochner-Riesz problem at critical index}
	\begin{abstract} In this paper we study maximal and square functions associated with bilinear Bochner-Riesz means at the critical index. In particular, we prove that they satisfy weighted estimates from $L^{p_1}(w_1)\times L^{p_2}(w_2)\rightarrow L^p(v_w)$ for bilinear weights $(w_1,w_2)\in A_{\vec{P}}$ where $p_1,p_2>1$ and $\frac{1}{p_1}+\frac{1}{p_2}=\frac{1}{p}$. Also, we show that both the operators fail to satisfy weak-type estimates at the end-point $(1,1,\frac{1}{2})$. 
	\end{abstract}	
	\maketitle
	\tableofcontents
\section{Introduction}\label{sec:intro}
Let $\alpha>0$ and consider the bilinear Bochner-Riesz mean defined by
\begin{eqnarray*}\label{def:bbr}\B^{\alpha}_R(f,g)(x)=\int_{\R^{n}}\int_{\R^{n}}\left(1-\frac{|\xi|^2+|\eta|^2}{R^2}\right)^{\alpha}_{+}\hat{f}(\xi)\hat{g}(\eta)e^{2\pi ix\cdot(\xi+\eta)}d\xi d\eta,
\end{eqnarray*}
where $R>0$ and $f,g\in \mathcal{S}(\R^n), n\geq 1$. Here $\hat f$ denotes the Fourier transform of $f$ given by $\hat{f}(\xi)=\int_{\R^n}f(x)e^{-2\pi ix.\xi}dx$ and $\mathcal{S}(\R^n)$ denotes the space of Schwartz class functions. 

We refer to ~\cite{BGSY, JLV, JS, JSK, LW} for the study of $L^p$ boundedness properties of the bilinear Bochner-Riesz means.  In this paper we are concerned with the maximal and square functions associated with the bilinear Bochner-Riesz means. The maximal function associated with the bilinear Bochner-Riesz means $\B^{\alpha}_R(f,g)(x)$ is defined by 
$$\B^{\alpha}_*(f,g)(x)=\sup_{R>0} |\B^{\alpha}_R(f,g)(x)|.$$
The maximal function $\B^{\alpha}_*(f,g)(x)$ plays a key role in addressing the issue of almost everywhere convergence of the bilinear Bochner-Riesz means $\B_R^{\alpha}(f,g)(x)$ as $R\rightarrow \infty$.  We refer to~\cite{JL, JS} for recent results on $L^p$ boundedness of the maximal function $\B^{\alpha}_*(f,g)$ for a wide range of $\alpha$ and exponents $p_1,p_2, p$. 

The bilinear analogue of Stein's square function for Bochner-Riesz means is recently introduced and studied in~\cite{CKSS}. This is defined by 
\begin{eqnarray*}
	\G^{\alpha}(f,g)(x)&:=&\left(\int_0^{\infty}|\frac{\partial}{\partial R}\mathcal{B}_R^{\alpha+1}(f,g)(x)|^2 RdR \right)^{1/2}\\
	&=&\left(\int_{0}^{\infty}|\mathcal K^{\alpha}_R\ast (f\otimes g)(x,x)|^2\frac{dR}{R}\right)^{1/2},
\end{eqnarray*}
where $\widehat{\mathcal K^{\alpha}_R}(\xi,\eta)=2(\alpha+1)\frac{|\xi|^2+|\eta|^2}{R^2}
\left(1-\frac{|\xi|^2+|\eta|^2}{R^2}\right)^{\alpha}_+$ and $\otimes$ denotes the tensor product. 

Note that in the spatial variables the kernel (in the sense of vector-valued operators) of $\G^{\alpha}$ is given by 
$${\mathcal K^{\alpha}_R}(y_1,y_2) = c_{n+\alpha}R^{2n-2}\Delta\left(\frac{ J_{\alpha+n} (|(Ry_1,Ry_2)|)} {|(Ry_1,Ry_2)|^{\alpha+n}}\right),~y_1,y_2\in \R^n.$$
Here $J_{\alpha+n}$ denotes the Bessel function of order $\alpha+n$. 

The index $\alpha=n-\frac{1}{2}$ is called the critical index for the bilinear Bochner-Riesz problem.  Motivated by the problem of linear Bochner-Riesz means at the critical index (which is $\frac{n-1}{2}$ for the linear case) and recent developments in the direction of bilinear Bochner-Riesz problem, see for example \cite{JS,JL,CKSS}, in this paper we investigate weighted boundedness of $\B^{n-\frac{1}{2}}_*$ and $\G^{n-\frac{1}{2}}$. Also, the issue of end-point boundedness for both the operators  is addressed. We invite the reader to \cite{Christ,St,SW,ana} for results on the linear Bochner-Riesz problem at critical index. 

We need to introduce some notation in order to state the results. Let $1\leqslant p_{1}, p_{2}< \infty$ and $p$ be such that $
\frac{1}{p}=\frac{1}{p_{1}}+\frac{1}{p_{2}}.$
\begin{definition}\cite[Definition 3.5]{Ler1} (Bilinear weights)
	Let $\vec{p}=(p_{1},p_{2})$. For a given pair of weights $\vec{w}=(w_{1},w_{2})$, set $
	v_w:=\prod_{i=1}^{2} w_{i}^{p/p_{i}}$.
	We say that $\vec{w}\in A_{\vec{P}}$ if 
	\begin{equation*}
		%\label{mulap}
		[\vec{w}]_{A_{\vec{P}}}:= \sup_{Q} \Big(\frac{1}{|Q|}\int_{Q}v_w\, dx \Big) \prod_{j=1}^{2}\Big(\frac{1}{|Q|}\int_{Q} w_{j}^{1-p'_{j}} \,dx\Big) ^{p/{p_{j}^\prime}}  <\infty.
	\end{equation*} 
	When $p_j=1$, $\big(\frac{1}{|Q|}\int_Qw_j^{1-p'_j}\big)^{1/p'_j}$ is understood as $(\inf_Q w_j)^{-1}$. Here $Q$ denotes a cube in $\R^n$ with sides parallel to coordinate axes. The quantity $[\vec{w}]_{A_{\vec{P}}}$ is referred to as the bilinear $A_{\vec{P}}$ characteristic of the bilinear weight 
	$\vec{w}$. 
\end{definition}

The following are the main results of this paper. 
\begin{theorem} \label{mainthm} Let $T=\mathcal {B}_*^{n-\frac{1}{2}}~~\text{or}~\G^{n-\frac{1}{2}}$. Then $T$ is bounded from  $L^{p_{1}}(\omega_{1})\times L^{p_{2}}(\omega_{2})\rightarrow L^{p}(v_{\omega})$ for all bilinear weights $\vec{\omega}\in A_{\vec{P}}$ with $1<p_{1}, p_{2} \leq \infty$ and $\frac{1}{p_{1}}+\frac{1}{p_{2}}=\frac{1}{p}$. 
\end{theorem}
Further, we show that both the operators fail to satisfy weak-type estimates at the end-point $(1,1,\frac{1}{2})$. 
\begin{proposition}\label{endpoint:sf}
The bilinear square function $\G^{n-\frac{1}{2}}$ is unbounded from $L^1(\R^n)\times L^1(\R^n)$ to $L^{\frac{1}{2},\infty}(\R^n), n\geq 1$.
\end{proposition}
In the case of maximal function we get a stronger result at the end-point $(1,1,\frac{1}{2})$. First, observe that in view of bilinear transference principle it is sufficient to work with the operator defined for functions on the unit cube $Q_n=[-\frac{1}{2},\frac{1}{2})^n$. Let us use the same notation to denote the operator in the periodic case as in the case of $\R^n$. We have the following. 
\begin{theorem}\label{div}
Let $n\geq 1$. There exists an integrable function $f$ on $Q_n$ and a positive measure set $E$ of $Q_n$ such that
$$\limsup_{R\to\infty}|\B^{{n-\frac{1}{2}}}_R(f,f)(x)|=\infty$$
for almost every $x\in E$.
\end{theorem}
In particular, we get that $\B_*^{n-\frac{1}{2}}$ is unbounded from $L^1(\R^n)\times L^1(\R^n)$ to $L^{\frac{1}{2},\infty}(\R^n), n\geq 1$. These results give us a complete picture of $L^p$ boundedness properties of operators $\B_*^{n-\frac{1}{2}}$ and $\G^{n-\frac{1}{2}}$. 

\subsection*{Organization of the paper} In Section~\ref{sec:mfexam} we prove Theorem~\ref{div}. The proof of Theorem~\ref{mainthm} establishing the weighted estimates for $\B^{n-\frac{1}{2}}_*$ and $\G^{n-\frac{1}{2}}$ is presented in~sections~\ref{max:sec} and \ref{sec:sqr} respectively. The issue of end-point isse for $\G^{n-\frac{1}{2}}$ is discussed in Section~\ref{sec:sqrexam}. 
\section{End-point estimates for the maximal function  \texorpdfstring{$\B^{n-\frac{1}{2}}_*$}{B}}\label{sec:mfexam}
In this section we prove Theorem~\ref{div}. We will make use of the ideas presented in [\cite{Bochner}, \cite{SW} page 267] where the corresponding result is proved for the maximal Bochner-Riesz function in the linear case.  We exploit their method and make suitable modifications to it to address the bilinear problem. 
%%%%%%%%%%%%%

The main idea is to estimate the maximal function acting on $L^1$ functions which peak at the origin. This is verified for the Dirac mass first. More precisely, first we show that $\B^{n-\frac{1}{2}}_R(\delta_0,\delta_0)$, as $R\rightarrow \infty$ becomes unbounded  for almost all $x\in Q_n$, where $\delta_0$ is the Dirac mass at the origin. This is proved in Lemma~\ref{Lemma:Dirac}. Later, we complete the proof in two steps. In the first step, with the help of the estimate~\eqref{Lemma:Dirac}, we replace one of the Dirac masses by a suitable $L^1$ function on $Q_n$. This step is then used to replace the other Dirac mass by the same $L^1$ function to achieve the desired result. 
\begin{lemma}\label{Lemma:Dirac}
Let $\delta_0$ be the Dirac mass at the origin in $\R^n, n\geq 1$. Then
\begin{eqnarray}\label{Dirac} \limsup_{R\to\infty}|\B^{n-\frac{1}{2}}_R(\delta_0,\delta_0)(x)|=\infty~
\end{eqnarray}
for almost every $x\in Q_n$. 
\end{lemma}

\begin{proof}
Consider the set 
$$S =\{x \in \R^{n}: \{1\}\cup\{|(x,x)-m|: m\in\Z^{2n}\}\text{ is linearly independent over  rationals}~ \Q\}.$$
Note that the complement of $S$ is a set of measure zero in $\R^n$. Let $x_0\in(Q_n\setminus\{0\})\cap S$. Let $K_R^{\alpha}(x,y)$ denote the kernel of the bilinear Bochner-Riesz operator $\B^{\alpha}$, i.e.,  $\hat{K}_R^{\alpha}(\xi,\eta)=(1-\frac{|m_1|^2+|m_2|^2}{R^2})^{\alpha}_{+}$, $(m_1,m_2)\in \Z^{2n}$. By the Poisson summation formula, we get 
\begin{equation}\label{kernel}
    K_R^\alpha((x_0,x_0))=C_\alpha R^{2n}\sum_{m\in\Z^{2n}}\frac{J_{n+\alpha}(2\pi R|(x_0,x_0)-m|)}{(R|(x_0,x_0)-m|)^{n+\alpha}}.
\end{equation}
Recall the asymptotics for the Bessel functions 
\begin{eqnarray*}
J_{n+\alpha}(2\pi R|(x_0,x_0)-m|)&=&\frac{e^{2\pi iR|(x_0,x_0)-m|}e^{-i(\frac{\pi}{2}(n+\alpha)+\frac{\pi}{4})}+e^{-2\pi iR|(x_0,x_0)-m|}e^{i(\frac{\pi}{2}(n+\alpha)+\frac{\pi}{4})}}{\pi\sqrt{R|(x_0,x_0)-m|}}\\
&&+~O((R|(x_0,x_0)-m|)^{-\frac{3}{2}}).
\end{eqnarray*}
Note that the infinite series in~\eqref{kernel} converges absolutely for $\alpha>n-\frac{1}{2}$. Recall that we are concerned with the estimates when $\alpha=n-\frac{1}{2}$, but we cannot take $\alpha\to n-\frac{1}{2}$ in the equation above. However, by taking an average over the parameter $R$, we can take $\alpha\to n-\frac{1}{2}$ and get the following estimate at $\alpha=n-\frac{1}{2}$. 
\begin{eqnarray*}
    \frac{1}{T}\int_1^T K_R^{n-\frac{1}{2}}((x_0,x_0))e^{2\pi i\lambda R}dR&=& C_n \sum_{m\in\Z^{2n}}\frac{e^{-i(\frac{\pi}{2}(2n-\frac{1}{2})+\frac{\pi}{4})}}{|(x_0,x_0)-m|^{2n}}\left(\frac{1}{T}\int_1^T e^{2\pi iR(\lambda+|(x_0,x_0)-m|)}dR\right) \\
    &&+~ C_n \sum_{m\in\Z^{2n}}\frac{e^{i(\frac{\pi}{2}(2n-\frac{1}{2})+\frac{\pi}{4})}}{|(x_0,x_0)-m|^{2n}}\left(\frac{1}{T}\int_1^T e^{-2\pi iR(\lambda-|(x_0,x_0)-m|)}dR\right)\\
    &&+~ C_n \sum_{m\in\Z^{2n}}O\left(\frac{1}{|(x_0,x_0)-m|^{2n+1}}\right)\frac{1}{T}\int_1^T \frac{dR}{R}.
\end{eqnarray*}
In the equation above note that if $\lambda\neq\pm|(x_0,x_0)-m|, m\in\Z^{2n}$, then all the terms on the right hand side vanish as $T\rightarrow \infty$. Moreover, if  $\lambda=\pm|(x_0,x_0)-m|$ for some $m\in\Z^{2n}$, the right hand side converges (as $T\rightarrow \infty$) to 
$$C_n \frac{e^{\pm in\pi}}{|(x_0,x_0)-m|^{2n}}.$$
Consider the set $\Lambda_{x_0}=\{|(x_0,x_0)-m|:m\in\Z^{2n}\}$ and enumerate it as $\Lambda_{x_0}=\{\lambda_1,\lambda_2,\lambda_3,...\},$ 
where $\lambda_1<\lambda_2<\lambda_3<...$ and $\sum\limits_{j=1}^\infty\frac{1}{\lambda_j^{2n}}=\infty$. With this choice and notation we have that 
\begin{align*}
    \lim_{T\to\infty}\frac{1}{T}\int_1^T K_R^{n-\frac{1}{2}}((x_0,x_0))e^{2\pi i\lambda R}dR=\left\{
	\begin{array}{ll}
		C_n \frac{e^{in\pi}}{\lambda_j^{2n}}  & \mbox{if } \lambda=\lambda_j \\
		0 & \mbox{if } \lambda\neq\pm\lambda_j \\
		C_n \frac{e^{-in\pi}}{\lambda_j^{2n}} & \mbox{if } \lambda=-\lambda_j
	\end{array}
\right.
\end{align*}
Since $(x_0,x_0)\in S$, the set $\{1\}\cup\{\lambda_1,\lambda_2,\lambda_3,...\}$ is linearly independent over the rationals $\mathbb Q$ and hence no expression of the form $\pm\lambda_{j_1},\pm\lambda_{j_2},...,\pm\lambda_{j_s}$ can be equal to an integer. This gives us that  
$$\lim_{T\to\infty}\frac{1}{T}\int_1^T K_R^{n-\frac{1}{2}}((x_0,x_0))\prod_{j=1}^N \left[1+\frac{e^{-in\pi}e^{2\pi i\lambda_j R}+e^{in\pi}e^{-2\pi i\lambda_j R}}{2}\right]dR=C_n\sum_{j=1}^N\frac{1}{\lambda_j^{2n}}$$
Note that the assumption 
$$\sup_{R\geq1}|K_R^{n-\frac{1}{2}}((x_0,x_0))|\leq A_{x_0}<\infty$$
will yield 
\begin{eqnarray*}
   C_n\sum_{j=1}^N\frac{1}{\lambda_j^{2n}}&=& \lim_{T\to\infty}\frac{1}{T}\int_1^T K_R^{n-\frac{1}{2}}((x_0,x_0))\prod_{j=1}^N\left[1+\frac{e^{-in\pi}e^{2\pi i\lambda_j R}+e^{in\pi}e^{-2\pi i\lambda_j R}}{2}\right]dR\\
   &\leq& A_{x_0} \frac{1}{T}\int_1^T\prod_{j=1}^N\left[1+\frac{e^{-in\pi}e^{2\pi i\lambda_j R}+e^{in\pi}e^{-2\pi i\lambda_j R}}{2}\right]dR=A_{x_0}
\end{eqnarray*}
This contradicts the choice that $\sum\limits_{j=1}^\infty\frac{1}{\lambda_j^{2n}}=\infty$. Therefore, for $x\in S\cap Q_n$ we get that 
$$\sup_{R\geq1}|K_R^{n-\frac{1}{2}}(x,x)|=\infty.$$
This completes the proof of Lemma~\ref{Lemma:Dirac}. 
\end{proof}
Next, we show that Dirac masses in Lemma~\ref{Lemma:Dirac} can be replaced with suitable $L^1$ functions. This part is done in two steps as follows. 
\subsection*{Step I:}
In this step we will show that in the estimate \eqref{Dirac} we can replace one of the Dirac masses with an $L^1$ function so that the estimate holds on a set of positive measure. 
 
Let $\Phi\in\mathcal{S}(\R^n)$ be a radial function such that $\hat{\Phi}$ is non-negative and supported in the unit ball of $\R^n$ with $\int_{\R^n}\hat{\Phi}(\xi)d\xi=1$. Given $\epsilon>0$ define  
$$\phi_\epsilon(x)=\frac{1}{\epsilon^n}\sum_{m\in\Z^n}\hat{\Phi}\left(\frac{x+m}{\epsilon}\right).$$

The Poisson summation formula yields 
$$\phi_\epsilon(x)=\frac{1}{\epsilon^n}\sum_{m\in\Z^n}\hat{\Phi}(\frac{x+m}{\epsilon})=\sum_{m\in\Z^n}{\Phi}(\epsilon m)e^{2\pi im\cdot x}.$$

Since $\Phi\in\mathcal{S}(\R^n)$ we get that  
$$\sum_{m\in\Z^n}|{\Phi}(\epsilon m)|\leq\sum_{m\in\Z^n}\frac{C'_n}{(1+\epsilon|m|)^{n+1}}\leq \frac{C_n}{\epsilon^{n}}.$$
Recall that the linear Bochner-Riesz means of order $n-\frac{1}{2}$ acting on Dirac mass  is given by 
\begin{eqnarray*}
B^{n-\frac{1}{2}}_R(\delta_0)(x)&=&\sum_{|m|\leq R}\left(1-\frac{|m|^2}{R^2}\right)^{n-\frac{1}{2}}e^{2\pi ix\cdot m}\\
&=&c_n \sum_{m\in\Z^n}R^n\frac{J_{\frac{3n-1}{2}}(2\pi R|x-m|)}{(R|x-m|)^{\frac{3n-1}{2}}}
\end{eqnarray*}
Observe that if $R\leq 10$ we have $|B^{n-\frac{1}{2}}_R(\delta_0)(x)|\leq C_1$ for all $x\in Q_n$. When $R>10$ and $|x|\geq\frac{1}{10}$, we have $R|x-m|\geq1$ for all $m\in\Z^n$. Therefore, using asymptotics of Bessel function for $x\in E=[\frac{1}{10},\frac{1}{2})\cup [-\frac{1}{2},\frac{1}{10}]$ we get that 
\begin{eqnarray*}
    \sup_{x\in E}\sup_{R>10}|B^{n-\frac{1}{2}}_R(\delta_0)(x)|&\leq&c_n\sup_{x\in E}\sup_{R>10}\sum_{m\in\Z^n}\frac{R^n\cos(2\pi R|x-m|-\frac{3n\pi}{2})}{(R|x-m|)^{\frac{3n}{2}}}
    \leq C_2
\end{eqnarray*}
Let $C=\max\{C_1,C_2\}$.

Next, we use an inductive argument to construct measurable subsets $E_j\subset E$ with $|E_j|\geq\frac{4}{5}-\frac{1}{j}$, an increasing sequence ${R_j}$ and two positive null sequences $\epsilon_j\leq\delta_j,~j\geq1$ such that
\begin{equation}\label{toprove}
    \sup_{R\leq R_j}|\B^{{n-\frac{1}{2}}}_R(f,\delta_0)(x)|\geq j \text{\hspace{10mm} for all $x\in E_j$},
\end{equation}
where $f=\sum\limits_{s=1}^\infty2^{-s}(\phi_{\epsilon_s}-\phi_{\delta_s})\in L^1(\R^n)$. 

Observe that the desired property  holds trivially with the initial choice of $E_1=\emptyset$, $R_1=1,$ and $\epsilon_1=\delta_1=1$.  Next, suppose that we have chosen $E_j,R_j,\epsilon_j,\delta_j$  satisfying \eqref{toprove} for all $1\leq j\leq k-1$. We need to construct $E_k,R_k,\epsilon_k$ and $\delta_k$ so that \eqref{toprove} holds for $j=k$. We will choose $\delta_k$ first. Let $B$ be a constant  such that
$$|\Phi(x)-\Phi(y)|\leq B|x-y|,~x,y\in\R^n.$$ 
Choose $\delta_k>0$ such that
\begin{equation*}\label{lipschitz}
    B\delta_k \sum_{|(m_1,m_2)|\leq R_{k-1}}|m_1|\leq 1.
\end{equation*}
Write  $A_k=CC_n\left(2^{-k}\delta_k^{-n}+\sum\limits_{s=1}^{k-1}2^{-s}(\epsilon_s^{-n}+\delta_s^{-n})\right)$. Consider  
\begin{eqnarray*}
    &&\B^{{n-\frac{1}{2}}}_R\left(-2^{-k}\phi_{\delta_k}+\sum_{j=1}^{k-1}2^{-s}(\phi_{\epsilon_s}-\phi_{\delta_s}),\delta_0\right)(x)\\
    &=&\nonumber\sum_{|(m_1,m_2)|\leq R}\left(1-\frac{|m_1|^2+|m_2|^2}{R^2}\right)^{{n-\frac{1}{2}}}\left(-2^{-k}\phi_{\delta_k}+\sum_{s=1}^{k-1}2^{-s}(\phi_{\epsilon_s}-\phi_{\delta_s})\right)^{\widehat{}}(m_1)e^{2\pi ix\cdot(m_1+m_2)}\\
    &=&\nonumber\sum_{|m_1|\leq R}\left(1-\frac{|m_1|^2}{R^2}\right)^{{n-\frac{1}{2}}}\left(-2^{-k}\Phi(\delta_k m_1)+\sum_{s=1}^{k-1}2^{-s}(\Phi(\epsilon_s m_1)-\Phi(\delta_s m_1))\right)e^{2\pi ix\cdot m_1}\\
    &&\nonumber\sum_{|m_2|\leq\sqrt{R^2-|m_1|^2}}\left(1-\frac{|m_2|^2}{R^2-|m_1|^2}\right)^{{n-\frac{1}{2}}}e^{2\pi ix\cdot m_2}.
\end{eqnarray*}
We make a crude estimate for the terms above in the following way. 
\begin{eqnarray}\label{smaller}
    &&\sup_{x\in E}\sup_{R>0}|\B^{{n-\frac{1}{2}}}_R\left(-2^{-k}\phi_{\delta_k}+\sum_{s=1}^{k-1}2^{-s}(\phi_{\epsilon_s}-\phi_{\delta_s}),\delta_0\right)(x)|\\
    &\leq& C \nonumber\sup_{R>0}\sum_{|m_1|\leq R}\left(2^{-k}|\Phi(\delta_k m_1)|+\sum_{s=1}^{k-1}2^{-s}(|\Phi(\epsilon_s m_1)|+|\Phi(\delta_s m_1)|)\right)\\
    &\leq&\nonumber C\sum_{m_1\in\Z^n}\left(2^{-k}|\Phi(\delta_k m_1)|+\sum_{s=1}^{k-1}2^{-s}(|\Phi(\epsilon_s m_1)|+|\Phi(\delta_s m_1)|)\right)\\
    &\leq&\nonumber CC_n\left(2^{-k}\delta_k^{-n}+\sum_{s=1}^{k-1}2^{-s}(\epsilon_s^{-n}+\delta_s^{-n})\right)= A_k
\end{eqnarray}
Using Fatou's lemma and the estimate proved in Lemma \ref{Lemma:Dirac}, we get that 
$$\liminf_{N\to\infty}\left|\left\{x\in E:\sup_{0<R\leq N} |\B^\alpha_R(\delta_0,\delta_0)(x)|>2^k(A_k+k+2)\right\}\right|=\frac{4}{5},$$
Thus, there exists an $R_k>R_{k-1}$ such that the set
$$E_k=\left\{x\in E:\sup_{0<R\leq R_k} |\B^\alpha_R(2^{-k}\delta_0,\delta_0)(x)|>A_k+k+2\right\}$$
has measure at least $\frac{4}{5}-\frac{1}{k}$.

Next, we choose $0<\epsilon_k\leq \delta_k$ so that
\begin{eqnarray*}
    &&\sup_{x\in Q_n}\sup_{R\leq R_k}|\B^{{n-\frac{1}{2}}}_R(2^{-k}\delta_0,\delta_0)(x)-\B^{{n-\frac{1}{2}}}_R(2^{-k}\phi_{\epsilon_k},\delta_0)(x)|\\
    &\leq& \sum_{|(m_1,m_2)|\leq R_k}2^{-k}\left(1-\frac{|m_1|^2+|m_2|^2}{R_k^2}\right)^{{n-\frac{1}{2}}}|1-\Phi(\epsilon_k m_1)|\leq 1
\end{eqnarray*}
Note that such a choice of $\epsilon_k$ is possible for a fixed $R_k$ because $|1-\Phi(\epsilon m_1)|\to 0$ as $\epsilon\to 0$. Therefore, we have 
\begin{equation}\label{equal}
    \inf_{x\in E_k}\sup_{R\leq R_k} |\B^{n-\frac{1}{2}}_R(2^{-k}\phi_{\epsilon_k},\delta_0)(x)|\geq A_k+k+1
\end{equation}
The choice of $\delta_k$ allows us to deduce the following estimate 
\begin{eqnarray}\label{greater}
   &&  \sup_{x\in Q_n}\sup_{R\leq R_k} \left|\B^{n-\frac{1}{2}}_R\left(\sum_{s=k+1}^\infty2^{-s}(\phi_{\epsilon_s}-\phi_{\delta_s}),\delta_0\right)(x)\right|\\
    &\leq&\nonumber \sum_{|(m_1,m_2)|\leq R_k}\left[\sum_{s=k+1}^\infty2^{-s}|\Phi(\epsilon_s m_1)-\Phi(\delta_s m_1)|\right]\\
    &\leq& \nonumber \sum_{|(m_1,m_2)|\leq R_k}\left[\sum_{s=k+1}^\infty2^{-s}B|\delta_s-\epsilon_s||m_1|\right]\\\nonumber
    &\leq& B\delta_{k+1} \sum_{|(m_1,m_2)|\leq R_k}|m_1|\leq 1.
\end{eqnarray}
Now for $j=k$ we have 
\begin{eqnarray*}
    \B^{{n-\frac{1}{2}}}_R\left(\sum_{s=1}^\infty2^{-s}(\phi_{\epsilon_s}-\phi_{\delta_s}),\delta_0\right)(x)&=&\B^{{n-\frac{1}{2}}}_R\left(-2^{-k}\phi_{\delta_k}+\sum_{s=1}^{k-1}2^{-s}(\phi_{\epsilon_s}-\phi_{\delta_s}),\delta_0\right)(x)\\
    &&+~ \B^{n-\frac{1}{2}}_R(2^{-k}\phi_{\epsilon_k},\delta_0)(x)+\B^{n-\frac{1}{2}}_R\left(\sum_{s=k+1}^\infty2^{-s}(\phi_{\epsilon_s}-\phi_{\delta_s}),\delta_0\right)(x)
\end{eqnarray*}
Using the estimates \eqref{smaller}, \eqref{equal} and \eqref{greater} for $x\in E_k$, we get that 
\begin{eqnarray*}
\sup_{R\leq R_k}|\B^{{n-\frac{1}{2}}}_R\left(\sum_{s=1}^\infty2^{-s}(\phi_{\epsilon_s}-\phi_{\delta_s}),\delta_0\right)(x)|&\geq&\sup_{R\leq R_k}|\B^{{n-\frac{1}{2}}}_R(2^{-k}\phi_{\epsilon_k},\delta_0)(x)|\\
&&-~\sup_{R\leq R_k}|\B^{{n-\frac{1}{2}}}_R\left(-2^{-k}\phi_{\delta_k}+\sum_{s=1}^{k-1}2^{-s}(\phi_{\epsilon_s}-\phi_{\delta_s}),\delta_0\right)(x)|\\
&&-~\sup_{R\leq R_k}|\B^{n-\frac{1}{2}}_R\left(\sum_{s=k+1}^\infty2^{-s}(\phi_{\epsilon_s}-\phi_{\delta_s}),\delta_0\right)(x)|\\
&\geq&k
\end{eqnarray*}

Denote $f=\sum_{s=1}^\infty2^{-s}(\phi_{\epsilon_s}-\phi_{\delta_s})$ and observe that we have  $\sup\limits_{R>0}|\B^{{n-\frac{1}{2}}}_R(f,\delta_0)(x)|\geq k$ for all $x\in\cup_{r\geq k}E_k$. Therefore, 
\begin{equation}\label{one1}
    \sup_{R>0}|\B^{n-\frac{1}{2}}_R(f,\delta_0)(x)|=\infty
\end{equation}
for all $x\in E=\cap_{r\geq1}\cup_{r\geq k}E_k$. Note that $|E|=\frac{4}{5}$.

\subsection*{Step II:} 
In this step we replace the Dirac mass in the second place by $f$ as constructed in the previous step. We need to make minor modifications to the arguments used in the previous step. We provide essential details here for a self contained proof. We will use the same notation as in the previous step. However, the parameters may differ from the previous step.

Let $M$ denote the classical Hardy-Littlewood maximal function defined by 
$$Mf(x):=\sup_{t>0}\frac{1}{|B(x,t)|}\int_{B(x,t)} |f(y)|dy,$$
where $B(x,t)$ is the euclidean ball of radius $t$ and center $x$. 
Since $f\in L^1(Q_n)$, we know that $M(f)(x)$ is finite a.e. $x\in Q_n$ and there holds  weak-type $(1,1)$ estimate  
$$|\{x\in E:|M(f)(x)|>N\}|\leq \frac{c_n}{N}\|f\|_1. $$
Choose $N$ large enough that $|\{x\in E:|M(f)(x)|>N\}|<\frac{1}{5}$. Let $F=\{x\in E:|M(f)(x)|\leq N\}$. Then $|F|\geq \frac{3}{5}$.

 For $j\geq1$, we will construct measurable subsets $F_j\subset F$ such that $|F_j|\geq\frac{3}{5}-\frac{1}{j}$, an increasing sequence ${R_j}$ and two positive null sequences $\epsilon_j\leq\delta_j$ such that
\begin{equation}\label{toprove1}
    \sup_{R\leq R_j}|\B^{{n-\frac{1}{2}}}_R\left(f,\sum_{s=1}^\infty2^{-s}(\phi_{\epsilon_s}-\phi_{\delta_s})\right)(x)|\geq j \text{\hspace{10mm} for all $x\in F_j$}.
\end{equation}

As previously, we begin with $F_1=\emptyset, R_1=1,$ and $\epsilon_1=\delta_1=1$. Suppose we have chosen $F_j,R_j,\epsilon_s,\delta_s$ for all $1\leq j\leq k-1$ satisfying \eqref{toprove1}. Choose $\delta_k>0$ small enough so that
\begin{equation*}\label{lipschitz1}
    B\|f\|_1\delta_k \sum_{|(m_1,m_2)|\leq R_{k-1}}|m_2|\leq 1.
\end{equation*}
Denote $B_k=NC_n\left(2^{-k}\delta_k^{-n}+\sum_{s=1}^{k-1}2^{-s}(\epsilon_s^{-n}+\delta_s^{-n})\right)$ and as in the previous step we get that
\begin{eqnarray}\label{smaller1}
    && \nonumber \sup_{x\in F}\sup_{R>0}|\B^{{n-\frac{1}{2}}}_R\left(f,-2^{-k}\phi_{\delta_k}+\sum_{s=1}^{k-1}2^{-s}(\phi_{\epsilon_s}-\phi_{\delta_s})\right)(x)|\\
    &\leq&\nonumber N~\sup_{R>0}\sum_{|m_2|\leq R}\left(2^{-k}|\Phi(\delta_k m_2)|+\sum_{s=1}^{k-1}2^{-s}(|\Phi(\epsilon_s m_2)|+|\Phi(\delta_s m_2)|)\right)\\
    &\leq&\nonumber N\sum_{m_2\in\Z^n}\left(2^{-k}|\Phi(\delta_k m_2)|+\sum_{s=1}^{k-1}2^{-s}(|\Phi(\epsilon_s m_2)|+|\Phi(\delta_s m_2)|)\right)\\
    &\leq&\nonumber NC_n\left(2^{-k}\delta_k^{-n}+\sum_{s=1}^{k-1}2^{-s}(\epsilon_s^{-n}+\delta_s^{-n})\right)\\
    &=& B_k.
\end{eqnarray}
Using Fatou's lemma and the estimate \eqref{one1}, we have
$$\liminf_{N\to\infty}\left|\left\{x\in F:\sup_{0<R\leq N} |\B^\alpha_R(f,\delta_0)(x)|>2^k(B_k+k+2)\right\}\right|=\frac{3}{5},$$
Choose $R_k>R_{k-1}$ such that the set 
$F_k=\left\{x\in F:\sup\limits_{0<R\leq R_k} |\B^\alpha_R(f,2^{-k}\delta_0)(x)|>B_k+k+2\right\}$
has measure at least $\frac{3}{5}-\frac{1}{k}$. Next, we choose $\epsilon_k\leq \delta_k$ so that
\begin{eqnarray*}
    &&\sup_{x\in F}\sup_{R\leq R_k}|\B^{{n-\frac{1}{2}}}_R(f,2^{-k}\delta_0)(x)-\B^{{n-\frac{1}{2}}}_R(f,2^{-k}\phi_{\epsilon_k})(x)|\\
    &\leq& \sum_{|(m_1,m_2)|\leq R_k}2^{-k}\left(1-\frac{|m_1|^2+|m_2|^2}{R_k^2}\right)^{{n-\frac{1}{2}}}|\hat{f}(m_1)||1-\Phi(\epsilon_k m_2)|\\
    &\leq& \sum_{|(m_1,m_2)|\leq R_k}2^{-k}\left(1-\frac{|m_1|^2+|m_2|^2}{R_k^2}\right)^{{n-\frac{1}{2}}}\|f\|_1|1-\Phi(\epsilon_k m_2)|\leq 1.
\end{eqnarray*}
Therefore, we get that 
\begin{equation}\label{equal1}
    \inf_{x\in F_k}\sup_{R\leq R_k} |\B^{n-\frac{1}{2}}_R(2^{-k}\phi_{\epsilon_k},\delta_0)(x)|\geq B_k+k+1.
\end{equation}
Also, we have 
\begin{eqnarray}\label{greater1}
    && \sup_{0<R\leq R_k} \left|\B^{n-\frac{1}{2}}_R\left(f,\sum_{s=k+1}^\infty2^{-s}(\phi_{\epsilon_s}-\phi_{\delta_s})\right)(x)\right|\\
    &\leq&\nonumber \sum_{|(m_1,m_2)|\leq R_k}|\hat{f}(m_1)|\left[\sum_{s=k+1}^\infty2^{-s}|\Phi(\epsilon_s m_2)-\Phi(\delta_s m_2)|\right]\\\nonumber
    &\leq& \sum_{|(m_1,m_2)|\leq R_k}\|f\|_1\left[\sum_{s=k+1}^\infty2^{-s}B|\delta_s-\epsilon_s||m_2|\right]\\\nonumber
    &\leq& B\|f\|_1\delta_{k+1} \sum_{|(m_1,m_2)|\leq R_k}|m_2| \leq 1.
\end{eqnarray}
When $j=k$, using the estimates \eqref{smaller1}, \eqref{equal1} and \eqref{greater1} for $x\in F$ we can get the following estimate (as in the previous step) 
\begin{eqnarray*}
\sup_{R\leq R_k}|\B^{{n-\frac{1}{2}}}_R\left(f,\sum_{s=1}^\infty2^{-s}(\phi_{\epsilon_s}-\phi_{\delta_s})\right)(x)|
&\geq&k
\end{eqnarray*}
This implies that $\sup\limits_{R>0}|\B^{{n-\frac{1}{2}}}_R(f,f)(x)|\geq k$ for all $x\in\cup_{r\geq k}F_k$. Therefore, 
\begin{equation*}\label{one}
    \sup_{R>0}|\B^{n-\frac{1}{2}}_R(f,f)(x)|=\infty
\end{equation*}
for all $x\in A=\cap_{r\geq1}\cup_{r\geq k}F_k$. Clearly, the set $A$ has positive measure. This completes the proof of Theorem~\ref{div}. \qed
%%%%%%%%%%%%%%%%%
\section{Weighted estimates for the maximal function \texorpdfstring{$\B^{n-\frac{1}{2}}_*$}{B1}}\label{max:sec}
The $L^p$ estimates for the maximal function $\B^{\alpha}_*$ were studied by~Grafakos, He and Honzik~\cite{GHH} and Jeong and Lee~\cite{JL}, which were later improved by Jotsaroop and Shrivastava~\cite{JS}. The problem of weighted boundedness of the bilinear Bochner-Riesz means $\B_R^{n-\frac{1}{2}}$ and the maximal function $\B^{n-\frac{1}{2}}_*$ was addressed in~\cite{JSK} for $n\geq 2$. The case of $n=1$ does not follow from their method. We complete the picture by giving a different proof of the weighted $L^p$ boundedness of $\B^{n-\frac{1}{2}}_*$. This proof works uniformly in all dimensions. We make use of the idea developed in~\cite{JS} to decompose the  bilinear Bochner-Riesz multiplier $(1-\frac{|\xi|^2+|\eta|^2}{R^2})^{\alpha}_{+}$ in a specific manner. This idea along with the Stein's complex interpolation for analytic family of bilinear operators is used to deduce the desired weighted estimates. This approach naturally requires us to consider the operator $\B^{\alpha}_*$ for complex parameter $\alpha$ which can be defined in a similar fashion by simply taking the multiplier $(1-\frac{|\xi|^2+|\eta|^2}{R^2})^{\alpha}_{+}$ for  $\alpha\in \C$ with $\text{Re}(\alpha)>0$.

The following lemma play a key role in proving Theorem~\ref{mainthm} for  $\B^{n-\frac{1}{2}}_*$. 
\begin{lemma}\label{keylem1}
Let $n\geq 1$ and $z$ be a complex number such that $0< Re(z)< n-\frac{1}{2}$. Then we have the following estimate 
\begin{eqnarray*} \int_{\mathbb{R}^n} |\partial_z \mathcal {B}^{z}_*(f,g)(x)| dx &\leq & C_{n+Re(z)} e^{\mathfrak{C} | Im(z)|^{2}}\| f\|_{L^{2}}\| g\|_{L^{2}},
\end{eqnarray*}
where $\mathfrak{C}>0$ is a constant. 
\end{lemma}
\begin{remark}
	Along with Lemma~\ref{keylem1} we will also require $L^2\times L^2\rightarrow L^1$ boundedness of the maximal bilinear Bochner-Riesz function $\mathcal {B}^{z}_*(f,g)$ from ~\cite{JS}. We will make use of the ideas developed in~\cite{JS} to prove Lemma~\ref{keylem1}. 
\end{remark}
We postpone the proof of Lemma~\ref{keylem1} to the next section and complete the proof of Theorem~\ref{mainthm} first. The following auxiliary results will be used in the proof of Theorem~\ref{mainthm}. 
\begin{lemma}\cite{JSK} \label{weighted} Let $n\geq 1$ and $z\in \mathbb C$ be such that $Re(z)>n-\frac{1}{2}.$ Then the operator $\mathcal {B}^{z}_*$ is bounded from $L^{p_{1}}(\omega_{1})\times L^{p_{2}}(\omega_{2})\rightarrow L^{p}(v_{\omega})$ for all $\vec{\omega}\in A_{\vec{P}}$ with $1<p_{1}, p_{2}<\infty$ and $\frac{1}{p_{1}}+\frac{1}{p_{2}}=\frac{1}{p}$.
\end{lemma} 
\begin{lemma}\label{wei1}\cite{Ler1}
	Let $\vec{\omega}=(\omega_{1},\omega_{2})\in A_{\vec{P}}$, where $\frac{1}{p}=\frac{1}{p_{1}}+\frac{1}{p_{2}}$ with $1<p_1,p_2 <\infty$, then there exists a $\delta>0$, such that $\vec{\omega}_{\delta}=(\omega^{1+\delta}_{1},\omega^{1+\delta}_{2})\in A_{\vec{P}}$. 
\end{lemma}
\subsection{Proof of Theorem~\ref{mainthm}}
First, note that in view of the multilinear extrapolation theorem, see~\cite{KJS} for details, it is enough to prove Theorem~\ref{mainthm} for $\vec{P}=(2,2).$ More precisely, we need to prove 
\begin{eqnarray}\label{mainthm1} 
	\|\mathcal {B}_*^{n-\frac{1}{2}}(f,g)\|_{L^1(v_{\omega})}\lesssim \|f\|_{L^{2}(\omega_{1})} \|g\|_{L^{2}(\omega_{2})} ~\text{for~all~}\vec{\omega}\in A_{\vec{P}}, ~\vec{P}=(2,2).
	\end{eqnarray} 
We linearize the maximal function using a standard trick. Let $R(x)$ be an arbitrary positive measurable function on $\R^n$ such that both $R(x)^{-1}$ and $R(x)$ are bounded. It is enough to prove the estimate ~\eqref{mainthm1} for $\mathcal {B}_{R(x)}^{n-\frac{1}{2}}(f,g)$ with bounds independent of the function $R(x)$. 

Fix such a function $R(x)$. Let $\epsilon_1,\epsilon_2>0$ (to be chosen later), $N\in \mathbb N$ and $A>\mathfrak{C}$, where $\mathfrak{C}$ is the constant appearing in Lemma~\ref{keylem1}. Consider the operator 
$$\tilde{\mathcal {B}}^{z,\epsilon_{1},\epsilon_{2},N}_{R(x)}(f,g)(x)=\mathcal {B}^{(1-z)\epsilon_{1}+z(n-\frac{1}{2}+\epsilon_{2})}_{R(x)}(f,g)(x)(v_{N}(x))^{z}e^{A z^{2}},$$ 
where 
$$ v_{N}(x)= \begin{cases} v_{\omega}(x), & \mbox{if }  v_{\omega}(x)\leq N \\ N, & \mbox{if }v_{\omega}(x)>N. \end{cases}$$
Note that $v_{N}(x)\leq v_{\omega}(x)$ a.e. $x$. 

Let $f$ and $g$ be compactly supported positive smooth functions. Given $\delta_{0}>0$ define $$f^{z}_{\delta_{0}}(x)=f(x)(\omega_{1}(x)+\delta_{0})^{-\frac{z}{2}}~~~\text{and}~~~g^{z}_{\delta_{0}}(x)=g(x)(\omega_{2}(x)+\delta_{0})^{-\frac{z}{2}}.$$
For $h\in L^\infty(\R^n)$ consider 
\begin{eqnarray*}
	\psi(z)&=&\int_{\mathbb{R}^{n}}\tilde{\mathcal {B}}_{R(x)}^{z,\epsilon_{1},\epsilon_{2},N}(f^{z}_{\delta_{0}},g^{z}_{\delta_{0}})(x)h(x)dx \\ &=&\int_{\mathbb{R}^{n}}\mathcal {B}^{(1-z)\epsilon_{1}+z(n-\frac{1}{2}+\epsilon_{2})}_{R(x)}(f^{z}_{\delta_{0}},g^{z}_{\delta_{0}})(x)(v_{N}(x))^{z}h(x)e^{A z^{2}}dx,
\end{eqnarray*}
where  $0\leq \text{Re}(z)\leq 1$. 

Use Lemma \ref{keylem1} to conclude that $\psi$ is analytic in the strip $S=\lbrace z\in \mathbb{C} : 0<Re(z)<1\rbrace$,  bounded and continuous on the closure $\bar{S}=\lbrace z\in \mathbb{C} : 0\leq Re(z)\leq 1\rbrace$. Moreover, we have the following estimates at the boundary. 
\begin{eqnarray*}
	\sup_{t\in\mathbb{R}}| \psi(i t)| 
	&\leq &  \| h\|_{\infty}\sup_{t\in\mathbb{R}}e^{-At^{2}}\int_{\R^n} | \mathcal {B}_{R(x)}^{(1-i t)\epsilon_{1}+i t(n-\frac{1}{2}+\epsilon_{2})}\left(f(\omega_{1}+\delta_{0})^{\frac{-i t}{2}}, g(\omega_{2}+\delta_{0})^{\frac{-i t}{2}} \right)(x)| dx\\
&& (\text{Since}~~ \text{Re}[(1-i t)\epsilon_{1}+i t(n-\frac{1}{2}+\epsilon_{2})]=\epsilon_{1}>0~\text{apply~Lemma}~\ref{keylem1}) \\
	&\leq &
	C_{\epsilon_{1},\epsilon_{2}}\| h\|_{\infty}\sup_{t\in\mathbb{R}}e^{-(A-\mathfrak{C})t^{2}}\left(\int_{\R^n} | f(\omega_{1}+\delta_{0})^{-\frac{i t}{2}}|^{2} dx \right)^{\frac{1}{2}}\left(\int_{\R^n} | g(\omega_{2}+\delta_{0})^{-\frac{i t}{2}}|^{2} dx \right)^{\frac{1}{2}}\\
	&\leq & C_{\epsilon_{1},\epsilon_{2}}\| h\|_{\infty}\| f\|_{2}\| g\|_{{2}}.
\end{eqnarray*}
Similarly, 
\begin{eqnarray*}
	& & \sup_{t\in\mathbb{R}} | \psi(1+i t)|\\
	& \leq & \| h\|_{\infty}\sup_{t\in \mathbb{R}}e^{A(1-t^{2})}\int_{\R^n} | \mathcal {B}_{R(x)}^{(-i t)\epsilon_{1}+(1+it)(n-\frac{1}{2}+\epsilon_{2})}\left(f(\omega_{1}+\delta_{0})^{-\frac{1+i t}{2}}, g(\omega_{2}+\delta_{0})^{-\frac{1+i t}{2}} \right)(x)| v_{N}(x)dx\\
	& \leq & \| h\|_{\infty}\sup_{t\in \mathbb{R}}e^{A(1-t^{2})}\int_{\R^n} | \mathcal {B}_{R(x)}^{(-i t)\epsilon_{1}+(1+i t)(n-\frac{1}{2}+\epsilon_{2})}\left(f(\omega_{1}+\delta_{0})^{-\frac{1+i t}{2}}, g(\omega_{2}+\delta_{0})^{-\frac{1+i t}{2}} \right)(x)| v_{\omega}(x)dx.
\end{eqnarray*} 
Note that 
$\text{Re}[(-it)\epsilon_{1}+(1+it)(n-\frac{1}{2}+\epsilon_{2})]=n-\frac{1}{2}+\epsilon_{2}>n-\frac{1}{2}$ and 
 $(\omega_{j}+\delta_{0})^{-1}\leq \omega^{-1}_{j}, ~j=0,1$. Therefore, applying  Lemma~\ref{weighted}, we get that 
\begin{eqnarray*}
	&& \sup_{t\in\mathbb{R}}| \psi(1+i t)|  \\
	&\leq & C_{\epsilon_{1},\epsilon_{2}}\| h\|_{\infty}\sup_{t\in\mathbb{R}}e^{-(A-\mathfrak C)t^{2}}\left(\int_{\R^n} | f(\omega_{1}+\delta_{0})^{-\frac{1+i t}{2}}|^{2}\omega_{1}(x) dx \right)^{\frac{1}{2}}\left(\int_{\R^n} | g(\omega_{2}+\delta_{0})^{-\frac{1+i t}{2}}|^{2}\omega_{2}(x) dx \right)^{\frac{1}{2}}\\
	& \leq & C_{\epsilon_{1},\epsilon_{2}}\| h\|_{\infty}\| f\|_{2}\| g\|_{2}.
\end{eqnarray*}
With these estimates on the boundary of the strip $S$ apply `Three lines lemma' from complex analysis to get that 
\begin{eqnarray*}
	| \psi(\theta)|
	&\leq & \nonumber C\left(\sup_{t\in \mathbb{R}}| \psi(\iota t)|\right)^{1-\theta}\left(\sup_{t\in \mathbb{R}}| \psi(1+\iota t)|\right)^{\theta}\\
	&\leq & \label{three} C_{\epsilon_{1},\epsilon_{2}}\| h\|_{\infty}\| f\|_{2}\| g\|_{2},~~0<\theta<1.
\end{eqnarray*} 
Note that the constant in the estimate above does not depend on $R(x)$. This gives us that 
\begin{eqnarray*}
	|\psi(\theta)|
	&=&\left|\int_{\mathbb{R}^{n}}\mathcal {B}_{R(x)}^{(1-\theta)\epsilon_{1}+\theta(n-\frac{1}{2}+\epsilon_{2})}(f^{\theta}_{\delta_{0}},g^{\theta}_{\delta_{0}})(x)(v_{N}(x))^{\theta}h(x)dx \right| \\
	&=&\left|\int_{\R^n} \mathcal {B}_{R(x)}^{(1-\theta)\epsilon_{1}+\theta(n-\frac{1}{2}+\epsilon_{2})}\left(f(\omega_{1}+\delta_{0})^{\frac{-\theta}{2}}, g(\omega_{2}+\delta_{0})^{\frac{-\theta}{2}} \right)(x)(v_{N}(x))^{\theta}h(x)dx \right|\\
	&\leq & C_{\epsilon_{1},\epsilon_{2}}\| f\|_{2}\| g\|_{2}.
\end{eqnarray*}
Since the constant $C$ in the inequality above is independent of $N$ and $\delta_{0}$, let $N\rightarrow \infty$ and $\delta_{0}\rightarrow0$ and then replace $f$ and $g$ by  $f\omega_{1}^{\frac{\theta}{2}}$ and $g\omega_{2}^{\frac{\theta}{2}}$ respectively to get that 
\begin{align} \label{estimate}
	\int_{\R^n} |\mathcal {B}_{R(x)}^{(1-\theta)\epsilon_{1} +\theta(n-\frac{1}{2}+\epsilon_{2})}\left(f, g \right)(x)|(v_{\omega}(x))^{\theta}dx
	&\leq & C \left(\int_{\R^n} | f(x)|^{2}\omega_{1}^{\theta} dx\right)^{\frac{1}{2}}\left(\int_{\R^n} | g(x)|^{2}\omega_{2}^{\theta} dx\right)^{\frac{1}{2}},
\end{align}
where $0<\theta<1$. 

At this point invoke the reverse H\"{older} inequality for bilinear weights from Lemma~\ref{wei1}. This tells us that given a bilinear weight $\vec{\omega}\in A_{\vec{P}}$ there exists $\delta>0$ such that $\vec{\omega}_{\delta}=(\omega_{1}^{1+\delta},\omega_{2}^{1+\delta})\in A_{\vec{P}}$. Using the estimate \eqref{estimate} for $\vec{\omega}_{\delta}\in A_{\vec{P}}$ with  $\theta=\frac{1}{1+\delta}$ we get the desired result. 
\begin{eqnarray*}\label{critical}
	\| \mathcal {B}_{R(x)}^{\lambda}(f,g)\|_{L^{1}(v_{\omega})}
	&\leq & C \| f\|_{L^{2}(\omega_{1})}\| g\|_{L^{2}(\omega_{2})},
\end{eqnarray*} 
where $\lambda=(1-\frac{1}{1+\delta})\epsilon_{1}+\frac{1}{1+\delta}(n-\frac{1}{2}+\epsilon_{2})$. 

Finally, observe that we can choose $\epsilon_1$ and $\epsilon_2$ so that $\lambda=n-\frac{1}{2}.$ 
This completes the proof of Theorem~\ref{mainthm}. 
\qed
\section{Proof of Lemma~\ref{keylem1}}\label{sec:keylem1}
Let $z\in \C$ be such that $0<\text{Re}(z)<n-\frac{1}{2}$. Choose functions $\psi\in C^\infty_0[\frac{1}{2},2]$ and $\psi_0\in C_0^{\infty}[-\frac{3}{4},\frac{3}{4}]$ such that
$$1=\sum_{j\geq 2}\psi(2^j(1-t))+\psi_0(t),~ t\in[0,1].$$ 
This gives us 
\begin{eqnarray*}
	m_R^{z}(\xi,\eta)
	=\left(1-\frac{|\xi|^2+|\eta|^2}{R^2}\right)_+^{z}
	= \sum\limits_{j\geq 2}{m}^{z}_{j,R}(\xi,\eta)+m^{z}_{0,R}(\xi,\eta),
\end{eqnarray*}
where $$m^{z}_{j,R}(\xi,\eta)=\psi\left(2^j\left(1-\frac{|\xi|^2}{R^2}\right)\right)\left(1-\frac{|\xi|^2}{R^2}\right)_+^{z}\left(1-\frac{|\eta|^2}{R^2}\left(1-\frac{|\xi|^2}{R^2}\right)^{-1}\right)^{z}_+$$ 
and 
$$m^{z}_{0,R}(\xi,\eta)=\psi_0\left(\frac{|\xi|^2}{R^2}\right)\left(1-\frac{|\xi|^2+|\eta|^2}{R^2}\right)_+^{z}.$$

Let $\mathcal{B}_{j,R}^{z}$ denote the bilinear multiplier operator associated with $m^{z}_{j,R}(\xi,\eta)$, i.e., 
\begin{equation*}\mathcal{B}_{j,R}^{z}(f,g)(x)=\int_{\R^n}\int_{\R^n}m^{z}_{j,R}(\xi,\eta)\hat{f}(\xi)\hat{g}(\eta)e^{2\pi ix\cdot (\xi+\eta)}d\xi d\eta.
\end{equation*}
Let $\beta>\frac{1}{2}$ and note that $\text{Re}(z)-\beta>-\frac{1}{2}$. Using the decomposition of the bilinear Bochner-Riesz multiplier from~Kaur and Shrivastava~[\cite{JS}, Section $3$] we have  the following representation 
\begin{eqnarray*}\label{decomope}
    \B^{z}_{j,R}(f,g)(x)
	&=& c_{z}\int_0^{\sqrt{2^{-j+1}}}S_{j,\beta}^{R,t}f(x)B_{Rt}^{z-\beta}g(x)t^{2(z-\beta)+1}dt,
\end{eqnarray*}
where $c_z=\frac{\Gamma(z+1)}{\Gamma(\beta)\Gamma(z-\beta+1)}$,  $B_{Rt}^{z-\beta}$ is the linear Bochner-Riesz mean and 
$${S}_{j,\beta}^{R,t}f(x)=\int_{\R^n}\psi\left(2^j\left(1-\frac{|\xi|^2}{R^2}\right)\right)\left(1-\frac{|\xi|^2}{R^2}-t^2\right)_+^{\beta-1}\hat{f}(\xi)e^{2\pi ix\cdot\xi}d\xi.$$ 
Therefore, we need to prove the desired boundedness results for  maximal functions  
$$\mathcal{B}_{j,*}^{z}(f,g)(x)=\sup_{R>0}|\mathcal{B}_{j,R}^{z}(f,g)(x)|$$
for $j=0$ and $j\geq2$. 

The derivative (with respect to $z$) of $\mathcal{B}_{j,R}^{z}(f,g)(x)$ is given by 
\begin{eqnarray*}\label{keylem1:der}
	\partial_z \mathcal {B}^{z}_{j,R}(f,g)(x)&=& I_{j,R}+II_{j,R}+III_{j,R}
	\end{eqnarray*}
where 
\begin{eqnarray*}  
I_{j,R}	&=& (\partial_zc_{z})\left(\int_0^{\sqrt{2^{-j+1}}}S_{j,\beta}^{R,t}f(x)B_{Rt}^{z-\beta}g(x)t^{2(z-\beta)+1}dt\right)
\end{eqnarray*}
	\begin{eqnarray*} II_{j,R}&=& c_{z}\int_0^{\sqrt{2^{-j+1}}}S_{j,\beta}^{R,t}f(x)\tilde{B}_{Rt}^{z-\beta}g(x)t^{2(z-\beta)+1}dt
	\end{eqnarray*}
\begin{eqnarray*} III_{j,R}&=&
	 c_{z}\int_0^{\sqrt{2^{-j+1}}}S_{j,\beta}^{R,t}f(x)B_{Rt}^{z-\beta}g(x)t^{2(z-\beta)+1}\log tdt
\end{eqnarray*}
where  $$\tilde{B}_t^{z-\beta}g(x)=\int_{\R^n}\hat{g}(\eta)\left(1-\frac{|\eta|^2}{t^2}\right)_+^{z-\beta}\log\left(1-\frac{|\eta|^2}{t^2}\right)_+ e^{2\pi ix.\eta} d\eta.$$

We will prove estimates for maximal functions associated with each of the terms above separately. Let us first record the bounds for constant $c_z$ and its derivative $\partial_zc_{z}$. 

Write $z=\alpha+i\tau$. Using estimates of gamma function, see~ [\cite{Grafakosclassical}, pages 569-570], we know that  
\begin{eqnarray*}
	|\Gamma(\alpha+1+i \tau)|\leq |\Gamma(\alpha+1)|\quad\text{and}\quad \frac{1}{|\Gamma(\alpha-\beta+1+i \tau)|}\leq \frac{e^{C_{\alpha,\beta}|\tau|^{2}}}{|\Gamma(\alpha-\beta+1)|},
\end{eqnarray*}
where $C_{\alpha,\beta}=\max\{(1+\alpha-\beta)^{-2},(1+\alpha-\beta)^{-1}\}$.
Therefore $|c_z|$ increases at most by a constant multiple of $e^{\mathfrak{C}|\tau|^2}$ when $0<\alpha<n-\frac{1}{2}$, where  $\mathfrak{C}$ is a fixed constant.  
Next we estimate the growth of $|\partial_zc_z|$. We have
$$\partial_zc_z=\frac{1}{\Gamma(\beta)}\left(\frac{\Gamma(z-\beta+1)\Gamma'(z+1)-\Gamma(z+1)\Gamma'(z-\beta+1)}{\Gamma(z-\beta+1)^2}\right).$$

It is easy to see that for $\text{Re}(z)>0$
\begin{eqnarray*}
|\Gamma'(z)|
&\lesssim& |\Gamma(Re(z)-\epsilon)|+|\Gamma(Re(z)+1)|,
\end{eqnarray*}
where $\epsilon$ is small enough so that $\text{Re}(z)-\epsilon>0$. Further, using estimates of gamma function we can show that  
\begin{eqnarray*}
|\partial_zc_z|
&\lesssim& \frac{|\Gamma(z-\beta+1)\Gamma'(z+1)|+|\Gamma(z+1)\Gamma'(z-\beta+1)|}{\Gamma(z-\beta+1)^2}\\
&\lesssim& \frac{|\Gamma(z-\beta+1)|(|\Gamma(Re(z)-\epsilon+1)|+|\Gamma(Re(z)+2)|)}{\Gamma(z-\beta+1)^2}\\
&&~+ \frac{|\Gamma(z-\beta+1)|(|\Gamma(Re(z)-\beta+1-\epsilon)||+|\Gamma(Re(z)-\beta+2)|}{\Gamma(z-\beta+1)^2}\\
&\lesssim& C_{Re(z)}e^{2C_{\alpha,\beta}|\tau|^{2}}.
\end{eqnarray*}
Therefore, $|\partial_zc_z|$ also increases at most by a constant multiple of $e^{\mathfrak{C}|\tau|^2}$.

\noindent {\bf Estimate for the term $I_{j,R}$:}~Note that $\Gamma(z+1)$ and $\frac{1}{\Gamma(z-\beta+1)}$ are analytic functions in the region  $0<Re(z)<n-\frac{1}{2}$. Applying Cauchy-Schwarz inequality twice and a change of variable argument in the second term we get that  

\begin{eqnarray*}
\left\|\sup_{R>0}|I_{j,R}|\right\|_1
&\leq&  2^{-\frac{j}{4}}|\partial_zc_{z}|\left\|\sup_{R>0}\left(\int_0^{\sqrt{2^{-j+1}}}|{S}_{j,\beta}^{R,t}f(x)t^{2({z-\beta})+1}|^2dt\right)^{1/2}\right\|_2\\
&& \left\|\sup_{R>0}\left(R_j^{-1}\int_0^{R_j}|B_t^{z-\beta}g(x)|^2dt\right)^{1/2}\right\|_2.
\end{eqnarray*}
Invoking Theorem 5.1 from~Kaur and Shrivastava~\cite{JS}) for $\beta>\frac{1}{2}$ we have 
$$\left\|\sup_{R>0}\left(\int_0^{\sqrt{2^{-j+1}}}|{S}_{j,\beta}^{R,t}f(x)t^{2(z-\beta)+1}|^2dt\right)^{1/2}\right\|_2\lesssim2^{j(\frac{1}{4}-Re(z)+\gamma)}\|f\|_2.$$
Also, the other operator satisfies the following $L^2-$estimate, see~[\cite{JS}, Lemma 4.3]. 
$$\left\|\sup_{R>0}\left(R_j^{-1}\int_0^{R_j}|B_t^{z-\beta}g(x)|^2dt\right)^{1/2}\right\|_2\lesssim \|f\|_2,~~\text{Re}({z-\beta})>-\frac{1}{2}.$$
Putting these estimates together we get that 
\begin{eqnarray*}
	\left\|\sup_{R>0}|I_{j,R}|\right\|_1
	&\lesssim& 2^{-j(\text{Re}(z)-\gamma)}e^{\mathfrak{C}|\text{Im}(z)|^2}\|f\|_2\|g\|_2.
\end{eqnarray*}
\noindent {\bf Estimate for the term $II_{j,R}$:}~~
As in the previous step the Cauchy-Schwarz inequality gives us 
\begin{eqnarray*}
\left\|\sup_{R>0}|II_{j,R}|\right\|_1
&\leq&  2^{-\frac{j}{4}}|c_z|\left\|\sup_{R>0}\left(\int_0^{\sqrt{2^{-j+1}}}|{S}_{j,\beta}^{R,t}f(x)t^{2(z-\beta)+1}|^2dt\right)^{1/2}\right\|_2\\
&&\nonumber \left\|\sup_{R>0}\left(R_j^{-1}\int_0^{R_j}|\tilde{B}_t^{z-\beta}g(x)|^2dt\right)^{1/2}\right\|_2.
\end{eqnarray*}
We already have the required bounds for the constant $c_z$ and the first term involving the operator  ${S}_{j,\beta}^{R,t}$. 

We claim that the following $L^2$ estimate holds 
\begin{eqnarray}\label{L2}
	\left\|\sup_{R>0}\left(R_j^{-1}\int_0^{R_j}|\tilde{B}_t^{\delta}g(x)|^2dt\right)^{1/2}\right\|_2 \lesssim  \|f\|_2,~~~\text{for~Re}(\delta)>-\frac{1}{2}.
\end{eqnarray}
Consequently, we get that 
\begin{eqnarray*}
	\left\|\sup_{R>0}|II_{j,R}|\right\|_1\lesssim  2^{-j(Re(z)-\gamma)}e^{\mathfrak{C}|Im(z)|^2}\|f\|_2\|g\|_2,~~~\text{Re}(z)-\beta>-\frac{1}{2}.
\end{eqnarray*}
We can use a trick from~Stein~\cite{SW} involving square function to prove~\eqref{L2}. Let $\text{Re}(\delta)>-\frac{1}{2}$ and choose $d$ such that $\text{Re}(\delta)+d>\frac{n-1}{2}$. Write 

$$\tilde{B}_t^{\delta}g=\sum_{k=1}^{d}\left(\tilde{B}_t^{\delta+k-1}g-\tilde{B}_t^{\delta+k}g\right) + \tilde{B}_t^{\delta+d}g.$$
		
This implies that  $$\left(\int_0^{R}|\tilde{B}_t^{\delta}g(x)|^2dt\right)^{1/2}\leq \sum_{k=1}^{d}\left(\int_0^{R}|\tilde{B}_t^{\delta+k}g(x)-\tilde{B}_t^{\delta+k-1}g(x)|^2dt\right)^{1/2} + \left(\int_0^{R}|\tilde{B}_t^{\delta+d}g(x)|^2dt\right)^{1/2}.$$
Observe that  $\sup\limits_{R>0}\left(R^{-1}\int_0^{R}|\tilde{B}_t^{\delta+k}g(x)-\tilde{B}_t^{\delta+k-1}g(x)|^2dt\right)^{1/2},~1\leq k\leq d$ is dominated by $\left(\int_0^{\infty}|\tilde{B}_t^{\delta+k}g(x)-\tilde{B}_t^{\delta+k-1}g(x)|^2 t^{-1}dt\right)^{1/2}.$ Using Plancherel's theorem, we get that 
\begin{eqnarray*}
&&\left\|\left(\int_0^{\infty}|\tilde{B}_t^{\delta+k}g(x)-\tilde{B}_t^{\delta+k-1}g(x)|^2 \frac{dt}{t}\right)^{1/2}\right\|_2^2\\
&=& \int_0^{\infty}\int_{\R^n}\left|\left(1-\frac{|\eta|^2}{t^2}\right)_+^{\delta+k}\frac{|\eta|^2}{t^2}\log\left(1-\frac{|\eta|^2}{t^2}\right)_+\hat{g}(\eta)\right|^2dx \frac{dt}{t}\\
&\lesssim & \int_0^{\infty}\int_{\R^n}\left|\left(1-\frac{|\eta|^2}{t^2}\right)_+^{\delta+k-\epsilon}\frac{|\eta|^2}{t^2}\hat{g}(\eta)\right|^2dx \frac{dt}{t}\\
%&=&c'\|G^{\delta+k-\epsilon}(g)\|_2^2\\
&\lesssim & \|g\|_2^2
\end{eqnarray*}
where we have chosen $\epsilon>0$ such that $\text{Re}(\delta)-\epsilon>-\frac{1}{2}$.

Finally, for the remaining term with $\text{Re}(\delta)+d>\frac{n-1}{2}$, one can easily verify that the kernel of $\tilde{B}_t^{\delta+d}$ is an integrable function. For, 
\begin{eqnarray*}
K_t^{\delta+d}(x)&=&\int_{\R^n}\left(1-\frac{|\eta|^2}{t^2}\right)_+^{\delta+d}\log\left(1-\frac{|\eta|^2}{t^2}\right)_+e^{2\pi ix\cdot\eta}d\eta\\
&=& t^n\int_0^\infty \int_{\mathbb{S}^{n-1}}(1-r^2)^{\delta+d}_+\log(1-r^2)_+e^{2\pi itx\cdot r\theta}d\theta r^{n-1}dr\\
&=& \frac{2\pi t^n}{|x|^\frac{n-2}{2}} \int_0^1 (1-r^2)^{\delta+d}_+\log(1-r^2)_+J_{\frac{n}{2}-1}(2\pi rt|x|)r^{\frac{n}{2}}dr\\
\end{eqnarray*}
Therefore, we get that 
\begin{eqnarray*}
|K_t^{\delta+d}(x)| &\leq& \frac{2\pi t^n}{|tx|^\frac{n-2}{2}}\sup_{r\in[0,1]}\{(1-r^2)^\epsilon_+\log(1-r^2)_+\} \int_0^1 (1-r^2)^{\delta+d-\epsilon}_+J_{\frac{n}{2}-1}(2\pi rt|x|)r^{\frac{n}{2}}dr\\
&=& Ct^n\frac{J_{\frac{n}{2}+\delta+d-\epsilon}(2\pi t|x|)}{|tx|^{\frac{n}{2}+\delta+d-\epsilon}},
\end{eqnarray*}
where $\epsilon$ is small enough so that $\delta+d-\epsilon>\frac{n-1}{2}$. Clearly, we get that 
$$\sup_{t>0}|\tilde{B}_t^{\delta+d}g(x)\leq c(\delta+d,n)Mg(x),$$
where $M$ is the classical Hardy-Littlewood maximal function. Consequently, we obtain the desired estimate~\eqref{L2}. \\

\noindent {\bf Estimate for the term $III_{j,R}$:}~~
In this case we have 
\begin{eqnarray*}
\left\|\sup_{R>0}|III_{j,R}|\right\|_1
&\leq&  2^{-j/4}|c_z|\left\|\sup_{R>0}\left(\int_0^{\sqrt{2^{-j+1}}}|{S}_{j,\beta}^{R,t}f(x)t^{2(z-\beta)+1}\log t|^2dt\right)^{1/2}\right\|_2\\
&&\nonumber \left\|\sup_{R>0}\left(R_j^{-1}\int_0^{R_j}|B_t^{z-\beta}g(x)|^2dt\right)^{1/2}\right\|_2.
\end{eqnarray*}
Following the discussion in the previous cases observe that we only need to deal with the term  involving ${S}_{j,\beta}^{R,t}$ in equation above. This can be done easily in the following manner. 

For $\text{Re}(z-\beta)>-\frac{1}{2}$, choose $\epsilon>0$ small enough so that $\text{Re}(2(z-\beta)+1)-\epsilon>0$. Then 
$$\sup_{R>0}\left(\int_0^{\sqrt{2^{-j+1}}}|{S}_{j,\beta}^{R,t}f(x)t^{2(z-\beta)+1}\log t|^2dt\right)^{1/2}\lesssim \sup_{R>0}\left(\int_0^{\sqrt{2^{-j+1}}}|{S}_{j,\beta}^{R,t}f(x)t^{2(z-\beta)+1-\epsilon}|^2dt\right)^{1/2}$$
as $t^\epsilon\log t$ is bounded in $[0,\sqrt{2^{-j+1}}]$. 
This reduces our job to a known situation as considered in~\cite{JS} and hence we have that  
$$\left\|\sup_{R>0}\left(\int_0^{\sqrt{2^{-j+1}}}|{S}_{j,\beta}^{R,t}f(x)t^{2(z-\beta)+1-\epsilon}|^2dt\right)^{1/2}\right\|_2\lesssim 2^{j(\frac{1}{4}-\text{Re}(z)+\epsilon+\gamma)}\|f\|_2,$$
for $\text{Re}(z)-\beta-\epsilon>-\frac{1}{2}$ and $\beta>\frac{1}{2}$. 
Putting the estimates above together, we get that 
$$\left\|\sup_{R>0}|III_{j,R}|\right\|_1\lesssim  2^{-j(\text{Re}(z)-\epsilon-\gamma)}e^{\mathfrak{C}|\text{Im}(z)|^2}\|f\|_2\|g\|_2.$$
Since $\text{Re}(z)>0$, we can sum over $j\geq 2$ in all the cases obtained as above. 

It remains to deal with the operator $\mathcal{B}_{0,*}^z$. \\

\noindent {\bf Boundedness of the operator $\mathcal{B}_{0,*}^z$:}
This part is dealt with similarly. We will make use of the decomposition of the multiplier $m_{0,R}^\alpha(\xi,\eta)$ as carried out Kaur and Shrivastava~[\cite{JS}, Section 5.3]. This time we decompose  $m_{0,R}^\alpha(\xi,\eta)$ with respect to $\eta$ variable. Note that previously we did the same with respect to $\xi$ variable. This gives us  

\begin{eqnarray*}
m_{0,R}^\alpha(\xi,\eta)&=&\sum_{j\geq0}\Tilde{m}_{j,R}^\alpha(\xi,\eta)
\end{eqnarray*}
where
$$\Tilde{m}_{j,R}^\alpha(\xi,\eta)=\psi_0\left(\frac{|\xi|^2}{R^2}\right)\psi\left(2^j\left(1-\frac{|\eta|^2}{R^2}\right)\right)\left(1-\frac{|\eta|^2}{R^2}\right)_+^{\alpha}\left(1-\frac{|\xi|^2}{R^2}\left(1-\frac{|\eta|^2}{R^2}\right)^{-1}\right)^{\alpha}_+$$
for $j\geq2$, and for $j=0,1$
$$\Tilde{m}_{j,R}^\alpha(\xi,\eta)=\psi_0\left(\frac{|\xi|^2}{R^2}\right)\psi_0^{j+1}\left(\frac{|\eta|^2}{R^2}\right)\left(1-\frac{|\xi|^2+|\eta|^2}{R^2}\right)_+^{\alpha}.$$
Here $\psi_0, \psi_0^1,$ and $ \psi_0^2$ are smooth functions supported in  $[0,3/4], [0,3/16],$ and $[\frac{3}{32},\frac{3}{4}]$  respectively. Also, they satisfy the identity $\psi_0(x)=\psi_0^1(x)+\psi_0^2(x)$. 

Let $\tilde{\B}^{\alpha}_{j,R}$ denote the bilinear multiplier operator associated with $\tilde{m}^{\alpha}_{j,R}(\xi,\eta)$ and let $\tilde{\B}^{\alpha}_{j,*}$ denote the corresponding maximal function. 

We will deal with maximal function $\tilde{\B}^{z}_{j*}, j\geq 0$ separately. 

Consider the case of $j=0$ first. In this case the multiplier is given by 
$$\tilde{m}_{0,R}^\alpha=\psi_0\left(\frac{|\xi|^2}{R^2}\right)\psi_0^1\left(\frac{|\eta|^2}{R^2}\right)\left(1-\frac{|\xi|^2+|\eta|^2}{R^2}\right)_+^{\alpha}.$$
Taking the derivative we see that  the bilinear multiplier for the operator $\left(\partial_z\right) \tilde{\B}^{z}_{0,R}$ is given by
$$M(\xi,\eta)=\psi_0\left(\frac{|\xi|^2}{R^2}\right)\psi_0^1\left(\frac{|\eta|^2}{R^2}\right)\left(1-\frac{|\xi|^2+|\eta|^2}{R^2}\right)_+^{z}\log\left(1-\frac{|\xi|^2+|\eta|^2}{R^2}\right)_+.$$
Observe that  $\psi_0\left(\frac{|\xi|^2}{R^2}\right)\psi_0^1\left(\frac{|\eta|^2}{R^2}\right)$ is a smooth function with its support in a ball of radius $\sqrt{\frac{15}{16}}R$ and observe that  $\left(1-\frac{|\xi|^2+|\eta|^2}{R^2}\right)_+^{z}\log\left(1-\frac{|\xi|^2+|\eta|^2}{R^2}\right)_+$ is a smooth function on this set. Therefore, using standard argument we can show that $\tilde{\B}^{z}_{0,*}$ is dominated by the bilinear Hardy-Littlewood maximal function which is defined by 
$$\M(f,g)(x):=\sup_{t>0}\frac{1}{|B(x,t)|^2}\int_{B(x,t)} |f(y)|dy\int_{B(x,t)} |g(z)|dz.$$
We refer to~\cite{Ler1} for more details about the maximal function $\M$. 
The $L^p$ boundedness of $\M$ yields the desired estimates for $\tilde{\B}^{z}_{0,*}$. 

Next, for $j=1$ Stein's identity [\cite{SW}, page 278] allows us to express 
$$\tilde{m}_{1,R}^\alpha(\xi,\eta)=\psi_0\left(\frac{|\xi|^2}{R^2}\right)\psi_0^2\left(\frac{|\eta|^2}{R^2}\right)R^{-2\alpha}\int_0^{uR}\left(R^2\varphi_R(\eta)-t^2\right)_+^{\beta-1}t^{2\delta+1}\left(1-\frac{|\xi|^2}{t^2}\right)^{\delta}_+dt,$$
where $\varphi_R(\eta)=\left(1-\frac{|\eta|^2}{R^2}\right)_+$ and $u=\sqrt{\frac{29}{32}}$. Therefore, 
\begin{eqnarray*} 
    \tilde{\B}^{z}_{1,R}(f,g)(x)&=&c_{z}\int_0^{u}H^{\beta}_{R,t}g(x)B_{Rt}^{z-\beta}f(x)t^{2(z-\beta)+1}dt,
\end{eqnarray*}
where $H_{R,t}^{\beta}g(x)=\int_{\R^n}\psi_0^2\left(\frac{|\eta|^2}{R^2}\right)\left(1-t^2-\frac{|\eta|^2}{R^2}\right)_+^{\beta-1}\hat{g}(\eta)e^{2\pi ix.\eta} d\eta$. 

This gives us that 
\begin{eqnarray}\label{j=1}
    \partial_z \tilde{\B}^{z}_{1,R}(f,g)(x)&=& \left(\partial_zc_{z}\right)\int_0^{u}H^{\beta}_{R,t}g(x)B_{Rt}^{z-\beta}f(x)t^{2(z-\beta)+1}dt\\
    &&\nonumber ~+ c_{z}\int_0^{u}H^{\beta}_{R,t}g(x)\tilde{B}_{Rt}^{z-\beta}f(x)t^{2(z-\beta)+1}dt\\
    &&\nonumber~+ c_{z}\int_0^{u}H^{\beta}_{R,t}g(x)B_{Rt}^{z-\beta}f(x)t^{2(z-\beta)+1}\log tdt.
\end{eqnarray}
Note that from this point onward the requires estimate can be deduced by following the corresponding argument (as in the case of $\B_{j,*}^z$) from the previous section along with $L^2$-estimate for the operator $f\rightarrow \left(\sup_{R>0}\int_0^{u}|H^{\beta}_{R,t}g(x)t^{2\delta+1}|^2 dt\right)^{1/2}$ from [\cite{JS}, Section 5.3]

Finally, when $j\geq2$, notice that $\tilde{m}^{\alpha}_{j,R}(\xi,\eta)$ is similar to $m^{\alpha}_{j,R}(\xi,\eta)$ except that there is an extra factor of $\psi_0\left(\frac{|\xi|^2}{R^2}\right)$ present in $\tilde{m}^{\alpha}_{j,R}(\xi,\eta)$. Let $K\in \N$ be such that $\{\xi: (\xi,\eta)\in \supp(\tilde{m}^{\alpha}_{j,R})\}\subseteq \{\xi:|\xi|\leq \frac{R}{8}\}, j\geq K$. We can assume that $\psi_0\left(\frac{|\xi|^2}{R^2}\right)= 1~\text{for}~|\xi|\leq \frac{R}{8}.$ Therefore, for $j\geq K$ the maximal function $\tilde{\B}^{\alpha}_{j,*}$ behaves the same way as $\B^{\alpha}_{j,*}$ and hence the desired results follow in this situation. 

For $2\leq j<K$, using Stein's identity once again we can write  
\begin{eqnarray*}
    \tilde{\B}^{z}_{j,R}(f,g)(x)
	&=& c_{z}\int_0^{\sqrt{2^{-j+1}}}S_{j,\beta}^{R,t}g(x)B^{\psi_0}_{R}B_{Rt}^{z-\beta}f(x)t^{2(z-\beta)+1}dt,
\end{eqnarray*}
where $c_z=\frac{\Gamma(z+1)}{\Gamma(\beta)\Gamma(z-\beta+1)}$ and $B^{\psi_0}_{R}f(x)=\int_{\R^n}\psi_0\left(\frac{|\xi|^2}{R^2}\right)\hat{f}(\xi)e^{2\pi ix.\xi} d\xi.$
Therefore, 
\begin{eqnarray*}
    \partial_z \tilde{\B}^{z}_{j,R}(f,g)(x)&=& \left(\partial_zc_{z}\right)\int_0^{\sqrt{2^{-j+1}}}S_{j,\beta}^{R,t}g(x)B^{\psi_0}_{R}B_{Rt}^{z-\beta}f(x)t^{2(z-\beta)+1}dt\\
    &&~+ c_{z}\int_0^{\sqrt{2^{-j+1}}}S_{j,\beta}^{R,t}g(x)B^{\psi_0}_{R}\tilde{B}_{Rt}^{z-\beta}f(x)t^{2(z-\beta)+1}dt\\
    &&~+ c_{z}\int_0^{\sqrt{2^{-j+1}}}S_{j,\beta}^{R,t}g(x)B^{\psi_0}_{R}B_{Rt}^{z-\beta}f(x)t^{2(z-\beta)+1}\log tdt
\end{eqnarray*}
The $L^2\times L^2\to L^1$-boundedness of the maximal functions associated with all the three terms can be proved similarly as in the previous case except that in place of estimate ~\eqref{L2} here we will require the following $L^2-$estimate 
\begin{eqnarray}\label{L21}
	\left\|\sup_{R>0}\left(R^{-1}\int_0^{R}|B^{\psi_0}_{R}\tilde{B}_t^{\delta}f(x)|^2dt\right)^{1/2}\right\|_2\lesssim \|f\|_2~~\text{for}~ ~\text{Re}(\delta)>-\frac{1}{2}.
	\end{eqnarray}
This is proved combining the estimate \eqref{L2} along with the estimate $$\sup\limits_{R>0}|B^{\psi_0}_{R}f(x)|\lesssim M(f)(x).$$ Note that the later assertion holds because $\psi_0$ is a compactly supported smooth function. 
This completes the proof of Lemma~\ref{keylem1}. 
\qed
\section{End-point estimates for the square function  \texorpdfstring{$\G^{n-\frac{1}{2}}$}{G}}\label{sec:sqrexam}
Recall that the kernel (in the sense of vector-valued operator) of square function $\G^{\alpha}$ is given by
\begin{eqnarray*}
{\mathcal K^{\alpha}_t}(y_1,y_2) &=& c_{n,\alpha}t^{2n-2}\Delta\left(\frac{ J_{\alpha+n} (|t(y_1,y_2)|)} {|t(y_1,y_2)|^{\alpha+n}}\right)\\
&=&c_{n,\alpha}t^{2n}\left(\frac{J_{n+\alpha}(2\pi |t(y_1,y_2)|)}{(|t(y_1,y_2)|)^{n+\alpha}}-\frac{J_{n+\alpha+1}(2\pi |t(y_1,y_2)|)}{(|t(y_1,y_2)|)^{n+\alpha+1}}\right).
\end{eqnarray*}

Using the asymptotics of Bessel functions for large $|x|$ and $t\geq1$, we have
$$J_{n+\alpha}(2\pi t|x|)=\frac{\cos(2\pi t|x|+\frac{\pi}{2}(n+\alpha)+\frac{\pi}{4})}{\pi\sqrt{t|x|}}+O((t|x|)^{-\frac{3}{2}}).$$
Then,
$$\mathcal{K}^{n-\frac{1}{2}}_t(y_1,y_2)=c_{n}\left(\frac{\cos(2\pi |t(y_1,y_2)|+n\pi)}{{|(y_1,y_2)|}^{2n}}+O\left(\frac{1}{|(y_1,y_2)|^{2n+1}}\right)\right).$$

Let $\psi\in \mathcal S(\R^n)$ be such that $\supp(\hat{\psi})$ is contained in $B(0,2)$ and $\hat{\psi}(\xi)=1$ in $B(0,1)$. Let $\psi_N(x)=N^n\psi(N^nx).$ Consider 
\begin{eqnarray*}
\G^{n-\frac{1}{2}}(\psi_N,\psi_N)(x) &\geq& \left(\int_1^N\left|\int_{\R^n\times\R^n}\left(1-\frac{|\xi|^2+|\eta|^2}{t^2}\right)^{n-\frac{1}{2}}_+\frac{|\xi|^2+|\eta|^2}{t^2}\hat{\psi}_N(\xi)\hat{\psi}_N(\eta)e^{2\pi ix\cdot(\xi+\eta)}d\xi\eta\right|^2\frac{dt}{t}\right)^{\frac{1}{2}}\\
&=& \left(\int_1^N\left|\K^{n-\frac{1}{2}}_t(x,x)\right|^2\frac{dt}{t}\right)^{\frac{1}{2}}\\
&\gtrsim & \frac{1}{|(x,x)|^{2n}}\left(\int_1^N|\cos(2\pi t|(x,x)|)|^2\frac{dt}{t}\right)^{\frac{1}{2}}\\
&= & \frac{1}{|(x,x)|^{2n}}\left(\int_{|(x,x)|}^{N|(x,x)|}|\cos(2\pi t)|^2\frac{dt}{t}\right)^{\frac{1}{2}}\\
&\gtrsim &  \frac{\log N}{|(x,x)|^{2n}}.
\end{eqnarray*}
Since $\|\psi_N\|_1=1$, we conclude that $\|\G^{n-\frac{1}{2}}(\psi_N,\psi_N)\|_{\frac{1}{2},\infty}\gtrsim \log N$. Therefore, $\G^{n-\frac{1}{2}}$ cannot be bounded from $L^1(\R^n)\times L^1(\R^n)$ to $L^{\frac{1}{2},\infty}(\R^n)$. This proves Proposition~\ref{endpoint:sf}.
\qed

\section{Weighted estimates for the square function \texorpdfstring{$\G^{n-\frac{1}{2}}$}{G1}}\label{sec:sqr}
In this section we prove Theorem~\ref{mainthm} for $T= \G^{n-\frac{1}{2}}$. The scheme of proof is exactly the same as in the case of maximal function $\B^{n-\frac{1}{2}}_*$. However, some of the estimates require different arguments. We will point out only the differences to avoid repetition. The $L^p$ boundedness of $\G^{\alpha}$ for a wide range exponents is proved in~\cite{CKSS}. We will exploit the techniques developed in~\cite{CKSS} to prove our proofs. 

First note that in view of the multilinear extrapolation theorem from \cite{KJS}, it is enough to prove the main Theorem~\ref{mainthm} for $\vec{P}=(2,2)$ and all weights in the corresponding class of bilinear weights. More precisely, we need to prove the following. 

\begin{theorem} \label{mainthm3} The bilinear Bochner-Riesz operator $\G^{n-\frac{1}{2}}$ is bounded from  $L^{2}(\omega_{1})\times L^{2}(\omega_{2})\rightarrow L^{1}(v_{\omega})$ for all $\vec{\omega}\in A_{\vec{P}},$ where $\vec{P}=(2,2).$
\end{theorem}
Moreover, following the discussion in Section~\ref{max:sec} in order to prove Theorem~\ref{mainthm3} we will require weighted estimates for $\G^{\alpha}$ when $\alpha>n-\frac{1}{2}$, $L^2\times L^2 \rightarrow L^1$ boundedness of $\G^{\alpha}$ for $0<\text{Re}(\alpha)<n-\frac{1}{2}$ and an analogue of the key Lemma~\ref{keylem1} in the context of square function. The weighted estimates for $\G^{\alpha}, \alpha>n-\frac{1}{2}$ are known from \cite{CKSS} as follows
\begin{theorem}\cite{CKSS}\label{weighted1}
Let $n\geq 1$ and $z\in \mathbb C$ be such that $\text{Re}(z)>n-\frac{1}{2}.$ Then the operator $\G^{z}$ is bounded from $L^{p_{1}}(\omega_{1})\times L^{p_{2}}(\omega_{2})\rightarrow L^{p}(v_{\omega})$ for all $\vec{\omega}\in A_{\vec{P}}$ with $1<p_{1}, p_{2}<\infty$ and $\frac{1}{p_{1}}+\frac{1}{p_{2}}=\frac{1}{p}$.
\end{theorem}
Also, $L^2\times L^2 \rightarrow L^1$ boundedness of $\G^{\alpha}$ for $0<\text{Re}(\alpha)<n-\frac{1}{2}$ has been obtained in~\cite{CKSS}.  Therefore, we need to establish the following analogue of the key lemma, Lemma~\ref{keylem1}.
\begin{lemma}\label{keylem2}
Let $n\geq 1$ and $z$ be a complex number such that $0< \text{Re}(z)< n-\frac{1}{2}$.  Then the following holds 
\begin{eqnarray*} \int_{\mathbb{R}^n} |\partial_z \G^{z}(f,g)(x)| dx &\leq & C_{n+\text{Re}(z)} e^{\mathfrak{C} | \text{Im}(z)|^{2}}\| f\|_{2}\| g\|_{2}, 
\end{eqnarray*}
where $\mathfrak{C}>0$ is a constant.
\end{lemma}
\noindent
{\bf Proof of Theorem~\ref{mainthm3} :} Assuming Lemma~\ref{keylem2} we follow the method of proof of Theorem~\ref{mainthm} for $T=\B^{n-\frac{1}{2}}_*$ to deduce the proof of Theorem~\ref{mainthm3}. 

Let $z=\alpha+i\tau\in \C$ and consider the linearized version of the square function 
\begin{equation*}
	\mathcal T^z_b(f,g)(x)=\int_0^\infty\int_{\R^n\times\R^n}\left(1-\frac{|\xi|^2+|\eta|^2}{R^2}\right)^{z}_{+}\frac{|\xi|^2+|\eta|^2}{R^2}\hat{f}(\xi)\hat{g}(\eta)e^{2\pi ix.(\xi+\eta)}d\xi d\eta b(x,R)\frac{dR}{R},
\end{equation*}
where $b(x,R)\in L^2((0,\infty),\frac{dR}{R})$ with $\int_0^\infty|b(x,R)|^2\frac{dR}{R}\leq 1.$ 

As in the proof of Theorem~\ref{mainthm} we let $\epsilon_1,\epsilon_2>0$ and $N\in \mathbb N$ and consider the operator 
$$\tilde{\mathcal {T}}^{z,\epsilon_{1},\epsilon_{2},N}_{b}(f,g)(x)=\mathcal {T}^{(1-z)\epsilon_{1}+z(n-\frac{1}{2}+\epsilon_{2})}_{b}(f,g)(x)(v_{N}(x))^{z}e^{A z^{2}},$$ 
such that $A>\mathfrak{C}$ and $v_{N}(x)$ is defined by
$$ v_{N}(x)= \begin{cases} v_{\omega}(x), & \mbox{if }  v_{\omega}(x)\leq N \\ N, & \mbox{if }v_{\omega}(x)>N. \end{cases}.$$
The proof from this point onwards may be completed imitating the method of proof of Theorem~\ref{mainthm} without any difficulty. We skip the details to avoid repetition. \qed
\section{Proof of Lemma~\ref{keylem2}}
We begin with the decomposition of the multiplier as previously. Also, see~[\cite{CKSS}, Section~3] for more details. We have  
\begin{eqnarray*}
	m_R^{\alpha}(\xi,\eta)
	&=&\left(1-\frac{|\xi|^2+|\eta|^2}{R^2}\right)_+^{\alpha}\frac{|\xi|^2+|\eta|^2}{R^2}
	= \sum\limits_{j\geq 2}{m}^{\alpha}_{j,R}(\xi,\eta)+m^{\alpha}_{0,R}(\xi,\eta),
\end{eqnarray*}
where $$m^{\alpha}_{j,R}(\xi,\eta)=\psi\left(2^j\left(1-\frac{|\xi|^2}{R^2}\right)\right)\left(1-\frac{|\xi|^2}{R^2}\right)_+^{\alpha}\left(1-\frac{|\eta|^2}{R^2}\left(1-\frac{|\xi|^2}{R^2}\right)^{-1}\right)^{\alpha}_+\frac{|\xi|^2+|\eta|^2}{R^2}$$ 
and 
$$m^{\alpha}_{0,R}(\xi,\eta)=\psi_0\left(\frac{|\xi|^2}{R^2}\right)\left(1-\frac{|\xi|^2+|\eta|^2}{R^2}\right)_+^{\alpha}\frac{|\xi|^2+|\eta|^2}{R^2}.$$

Let $\g_R^{\alpha}$ denote the bilinear operator associated with the multiplier $m^{\alpha}_{j,R}(\xi,\eta)$ and $\G_j^{\alpha}$ denote the corresponding bilinear square function. Then, we have 
\begin{eqnarray*}\label{firstdecom}	\G^{\alpha}(f,g)(x)\leq 	\G_0^{\alpha}(f,g)(x)+\sum_{j\geq 2}	\G_j^{\alpha}(f,g)(x).
	\end{eqnarray*}
Following the decomposition of the multiplier $m^{\alpha}_{j,R}(\xi,\eta)$ from~[\cite{CKSS}, equation $(7)$] we can write
\begin{eqnarray*} 
	\g^{\alpha}_{j,R}(f,g)(x)
	&=&c_{\alpha}\int_0^{\sqrt{2^{-j+1}}}(S_{j,\beta}^{R,t}f(x)A_{Rt}^{\delta}g(x)+\tilde{S}_{j,\beta}^{R,t}f(x)B_{Rt}^{\delta}g(x))t^{2\delta+1}dt,
\end{eqnarray*}
where $\beta>\frac{1}{2}$, $\delta>-\frac{1}{2}$ and $\beta+ \delta=\alpha$ and 
\begin{eqnarray*}B_t^{\delta}g(x)=\int_{\R^n}\hat{g}(\eta)\left(1-\frac{|\eta|^2}{t^2}\right)^{\delta}_+e^{2\pi ix\cdot\eta} d\eta,
\end{eqnarray*}

\begin{eqnarray*}A_t^{\delta}g(x)=\int_{\R^n}\hat{g}(\eta)\frac{|\eta|^2}{t^2}\left(1-\frac{|\eta|^2}{t^2}\right)^{\delta}_+e^{2\pi ix\cdot\eta} d\eta,
\end{eqnarray*}
$${S}_{j,\beta}^{R,t}f(x)=\int_{\R^n}\psi\left(2^j\left(1-\frac{|\xi|^2}{R^2}\right)\right)\left(1-\frac{|\xi|^2}{R^2}-t^2\right)_+^{\beta-1}\hat{f}(\xi)e^{2\pi ix\cdot\xi}d\xi,$$
and
$$\tilde{S}_{j,\beta}^{R,t}f(x)=\int_{\R^n}\psi\left(2^j\left(1-\frac{|\xi|^2}{R^2}\right)\right)\frac{|\xi|^2}{R^2}\left(1-\frac{|\xi|^2}{R^2}-t^2\right)_+^{\beta-1}\hat{f}(\xi)e^{2\pi ix\cdot\xi}d\xi.$$
The same decomposition can be preformed for the multiplier with complex exponent. This gives us the following representation of  $g^{z}_{j,R}(f,g)(x)$ for $z\in \C$  with $0<Re(z)<n-\frac{1}{2}$
\begin{eqnarray*}\label{decomope1}
    \g^{z}_{j,R}(f,g)(x)
	&=& c_{z}\int_0^{\sqrt{2^{-j+1}}}[S_{j,\beta}^{R,t}f(x)A_{Rt}^{z-\beta}g(x)+\tilde{S}_{j,\beta}^{R,t}f(x)B_{Rt}^{z-\beta}g(x)]t^{2(z-\beta)+1}dt,
\end{eqnarray*}
where $\beta>\frac{1}{2}, \text{Re}(z)-\beta>-\frac{1}{2}$ and $c_z=\frac{\Gamma(z+1)}{\Gamma(\beta)\Gamma(z-\beta+1)}$, see \cite[page $279$]{SW} for the precise form of the constant $c_z$. 
\subsection*{Boundedness of $\G^{\alpha}_j, j\geq 2$:} 
The derivative of $\mathcal \g^{z}_{j,R}(f,g)(x)$ is given by 
\begin{eqnarray*}
    \left(\partial_z\right) \mathcal \g^{z}_{j,R}(f,g)(x)
    &=& \left(\partial_zc_{z}\right)\int_0^{\sqrt{2^{-j+1}}}[S_{j,\beta}^{R,t}f(x)A_{Rt}^{z-\beta}g(x)+\tilde{S}_{j,\beta}^{R,t}f(x)B_{Rt}^{z-\beta}g(x)]t^{2(z-\beta)+1}dt\\
    &&~+ c_{z}\int_0^{\sqrt{2^{-j+1}}}[S_{j,\beta}^{R,t}f(x)\tilde{A}_{Rt}^{z-\beta}g(x)+\tilde{S}_{j,\beta}^{R,t}f(x)\tilde{B}_{Rt}^{z-\beta}g(x)]t^{2(z-\beta)+1}dt\\
    &&~+ c_{z}\int_0^{\sqrt{2^{-j+1}}}[S_{j,\beta}^{R,t}f(x)A_{Rt}^{z-\beta}g(x)+\tilde{S}_{j,\beta}^{R,t}f(x)B_{Rt}^{z-\beta}g(x)]t^{2(z-\beta)+1}\log tdt\\
    &=& I_{j,R}+II_{j,R}+III_{j,R}
\end{eqnarray*}
where  $$\tilde{B}_t^{z-\beta}g(x)=\int_{\R^n}\hat{g}(\eta)\left(1-\frac{|\eta|^2}{t^2}\right)_+^{z-\beta}\log\left(1-\frac{|\eta|^2}{t^2}\right)_+ e^{2\pi ix.\eta} d\eta$$ and $$\tilde{A}_t^{z-\beta}g(x)=\int_{\R^n}\hat{g}(\eta)\left(1-\frac{|\eta|^2}{t^2}\right)_+^{z-\beta}\frac{|\eta|^2}{t^2}\log\left(1-\frac{|\eta|^2}{t^2}\right)_+ e^{2\pi ix.\eta} d\eta.$$ 

From the proof of Lemma~\ref{keylem1} we know that constants $|c_z|$ and $|\partial_zc_z|$ increase at most by a constant multiple of $e^{\mathfrak{C}\tau^2}$, where $\mathfrak{C}>0$ is a fixed constant. 

Next, we need to prove the desired estimates for each of the three square functions associated with quantities $I_{j,R}, II_{j,R}$ and $III_{j,R}.$ The proof of these estimates may be completed following the scheme of proof for the operator $\partial_z \mathcal {B}^{z}_{j,R}(f,g)$ as in Section~\ref{sec:keylem1}. Of course, we will have to make minor modifications in the arguments, but this part can be completed without much difficulty imitating the proof of its counterpart in Lemma~\ref{keylem1}. In doing so we will require [\cite{JS}, Theorem 5.1], [\cite{CKSS}, Theorem 3.2] and the following proposition. 
\begin{proposition}\label{sqf2}
The operator
$$g\to\left(\int_0^\infty\int_{0}^{\sqrt{2^{-j+1}}}|\tilde{A}_{Rt}^{\delta}g(x)|^2dt\frac{dR}{R}\right)^{\frac{1}{2}}$$
is bounded on $L^2(\R^n)$ for $\text{Re}(\delta)>-\frac{1}{2}$
\end{proposition}
\begin{proof}
Consider 
\begin{eqnarray*}
    \left\|\left[\int_0^\infty\left(\int_{0}^{\sqrt{2^{-j+1}}}|\tilde{A}_{Rt}^{\delta}g(x)|^2dt\right)\frac{dR}{R}\right]^{\frac{1}{2}}\right\|_2 
	&=& \left\|\left[\int_{0}^{\sqrt{2^{-j+1}}}\left(\int_0^\infty|\tilde{A}_{R}^{\delta}g(x)|^2\frac{dR}{R}\right)t^4dt\right]^{\frac{1}{2}}\right\|_2\\
	&\lesssim& 2^{\frac{5}{4}(-j+1)}\left\|\left[\int_0^\infty|\tilde{A}_{R}^{\delta}g(x)|^2\frac{dR}{R}\right]^{\frac{1}{2}}\right\|_2.
\end{eqnarray*}

Use Plancherel's theorem to deduce that 
\begin{eqnarray*}
\left\|\left(\int_0^{\infty}|\tilde{A}_R^{\delta}g(x)|^2 \frac{dR}{R}\right)^{1/2}\right\|_2^2&=&\int_0^{\infty}\int_{\R^n}\left|\left(1-\frac{|\eta|^2}{R^2}\right)_+^{\delta}\frac{|\eta|^2}{R^2}\log\left(1-\frac{|\eta|^2}{t^2}\right)_+\hat{g}(\eta)\right|^2dx\frac{dR}{R}\\
&\lesssim & \int_0^{\infty}\int_{\R^n}\left|\left(1-\frac{|\eta|^2}{R^2}\right)_+^{\delta+k-\epsilon}\frac{|\eta|^2}{R^2}\hat{g}(\eta)\right|^2dx \frac{dR}{R}\\
%&=&c'\|G^{\delta+k-\epsilon}(g)\|_2^2\\
&\lesssim& \|g\|_2^2
\end{eqnarray*}
where $\epsilon>0$ is small enough so that $\delta-\epsilon>-\frac{1}{2}$.
\end{proof}
\noindent {\bf Boundedness for $\G_0^\alpha(f,g)$:} Finally, we need to prove $L^2\times L^2\to L^1$-boundedness of the square function $\G_0^\alpha(f,g)$. 
%We decompose the multiplier $m_{0,R}^\alpha(\xi,\eta)$ using the partition of unity in the $\eta$ variable same as we did earlier in the $\xi$ variable. We get that 
%\begin{align*} m_{0,R}^\alpha(\xi,\eta)=\sum_{j\geq 2}\psi_0\left(\frac{|\xi|^2}{R^2}\right)\psi\left(2^j\left(1-\frac{|\eta|^2}{R^2}\right)\right)\frac{|\xi|^2+|\eta|^2}{R^2}\left(1-\frac{|\eta|^2}{R^2}\right)_+^{\alpha}\left(1-\frac{|\xi|^2}{R^2}\left(1-\frac{|\eta|^2}{R^2}\right)^{-1}\right)^{\alpha}_+\nonumber \\+\psi_0\left(\frac{|\xi|^2}{R^2}\right)\psi_0\left(\frac{|\eta|^2}{R^2}\right)\frac{|\xi|^2+|\eta|^2}{R^2}\left(1-\frac{|\xi|^2+|\eta|^2}{R^2}\right)_+^{\alpha}.
%\end{align*}
%
%Since the support of the function $\psi_0$ is contained in the set $[0,3/4]$, the function \\$\psi_0\left(\frac{|\xi|^2}{R^2}\right)\psi_0\left(\frac{|\eta|^2}{R^2}\right)\frac{|\xi|^2+|\eta|^2}{R^2}\left(1-\frac{|\xi|^2+|\eta|^2}{R^2}\right)_+^{\alpha}$ is not smooth. So, we further split $\psi_0(x)=\psi_0^1(x)+\psi_0^2(x)$ such that support of $\psi_0^1$ is contained in the interval $[0,3/16]$ and support of $\psi_0^2$ is contained in the interval $[\frac{3}{32},\frac{3}{4}]$. We get the following decomposition of $m_{0,R}^\alpha$,
As earlier we decompose the multiplier in the following manner. 
\begin{eqnarray*}
m_{0,R}^\alpha(\xi,\eta)&=&\sum_{j\geq0}\Tilde{m}_{j,R}^\alpha(\xi,\eta)
\end{eqnarray*}
where
$$\Tilde{m}_{j,R}^\alpha(\xi,\eta)=\psi_0\left(\frac{|\xi|^2}{R^2}\right)\psi\left(2^j\left(1-\frac{|\eta|^2}{R^2}\right)\right)\frac{|\xi|^2+|\eta|^2}{R^2}\left(1-\frac{|\eta|^2}{R^2}\right)_+^{\alpha}\left(1-\frac{|\xi|^2}{R^2}\left(1-\frac{|\eta|^2}{R^2}\right)^{-1}\right)^{\alpha}_+$$
for $j\geq2$, and for $j=0,1$
$$\Tilde{m}_{j,R}^\alpha(\xi,\eta)=\psi_0\left(\frac{|\xi|^2}{R^2}\right)\psi_0^{j+1}\left(\frac{|\eta|^2}{R^2}\right)\frac{|\xi|^2+|\eta|^2}{R^2}\left(1-\frac{|\xi|^2+|\eta|^2}{R^2}\right)_+^{\alpha}.$$

Let $\tilde{\g}^{\alpha}_{j,R}$ denote the bilinear multiplier operator associated with $\tilde{m}^{\alpha}_{j,R}(\xi,\eta)$ and let $\tilde{\G}^{\alpha}_{j}$ be the square function corresponding to $\tilde{\g}^{\alpha}_{j,R}$. 

Note that for $j=0$ the multiplier 
$M(\xi,\eta)=\psi_0\left(\frac{|\xi|^2}{R^2}\right)\psi_0^1\left(\frac{|\eta|^2}{R^2}\right)\left(1-\frac{|\xi|^2+|\eta|^2}{R^2}\right)_+^{z}\frac{|\xi|^2+|\eta|^2}{R^2}$ $\log\left(1-\frac{|\xi|^2+|\eta|^2}{R^2}\right)_+$
is a smooth function with its support in a ball of radius $\sqrt{\frac{15}{16}}R$. Consequently, this part can be dominated by the bilinear Hardy-Littlewood maximal function and the desired estimate follows. 

When $j=1$ similar to the expression~\eqref{j=1} we get that  
\begin{eqnarray*}
	\partial_z\left( \tilde{\g}^{z}_{1,R}(f,g)(x)\right)&=& \left(\partial_zc_{z}\right)\int_0^{u}[H^{\beta}_{R,t}g(x)A_{Rt}^{z-\beta}f(x)+\tilde{H}^{\beta}_{R,t}g(x)B_{Rt}^{z-\beta}f(x)]t^{2(z-\beta)+1}dt\\
	&&~+ c_{z}\int_0^{u}[H^{\beta}_{R,t}g(x)\tilde{A}_{Rt}^{z-\beta}f(x)+\tilde{H}^{\beta}_{R,t}g(x)\tilde{B}_{Rt}^{z-\beta}f(x)t]^{2(z-\beta)+1}dt\\
	&&~+ c_{z}\int_0^{u}[H^{\beta}_{R,t}g(x)A_{Rt}^{z-\beta}f(x)+\tilde{H}^{\beta}_{R,t}g(x)B_{Rt}^{z-\beta}f(x)]t^{2(z-\beta)+1}\log tdt
\end{eqnarray*}
where $u=\sqrt{\frac{29}{32}}$,   $$H_{R,t}^{\beta}g(x)=\int_{\R^n}\psi_0^2\left(\frac{|\eta|^2}{R^2}\right)\left(1-t^2-\frac{|\eta|^2}{R^2}\right)_+^{\beta-1}\hat{g}(\eta)e^{2\pi ix.\eta} d\eta,$$ and
$$\tilde{H}_{R,t}^{\beta}g(x)=\int_{\R^n}\psi_0^2\left(\frac{|\eta|^2}{R^2}\right)\left(1-t^2-\frac{|\eta|^2}{R^2}\right)_+^{\beta-1}\frac{|\eta|^2}{R^2}\hat{g}(\eta)e^{2\pi ix.\eta} d\eta.$$
It follows that the bilinear square function associated to each of the terms in the equation above is  bounded from $L^2\times L^2$ into $L^1$. Here we need to use $L^2$-boundedness of operators  $g\rightarrow \left(\sup_{R>0}\int_0^{u}|H^{\beta}_{R,t}g(x)t^{2\delta+1}|^2 dt\right)^{1/2}$ and $g\rightarrow \left(\int_0^\infty\int_0^{u}|\tilde{H}^{\beta}_{R,t}g(x)t^{2\delta+1}|^2 dt\frac{dR}{R}\right)^{1/2}$ from [\cite{JS}, Section 5.3] and [\cite{CKSS}, page 16] respectively.

Finally, consider the case $j\geq2$. Let $K\in \N$ be such that $\{\xi: (\xi,\eta)\in \supp(\tilde{m}^{\alpha}_{j,R})\}\subseteq \{\xi:|\xi|\leq \frac{R}{8}\}$ for all $j\geq K$. As earlier we may  assume that $\psi_0\left(\frac{|\xi|^2}{R^2}\right)= 1$ for all $|\xi|\leq \frac{R}{8}.$ Therefore, for $j\geq K$ the square function $\tilde{\G}^{\alpha}_{j}$ can be dealt with exactly the same way as $\G^{\alpha}_{j}$. For the remaining terms, i.e., for $2\leq j<K$ using the approach similar to the expression \eqref{decomope1}, we get that 
\begin{eqnarray*}
    \tilde{\g}^{z}_{j,R}(f,g)(x)
	&=& c_{z}\int_0^{\sqrt{2^{-j+1}}}[S_{j,\beta}^{R,t}g(x)B^{\psi_0}_{R}A_{Rt}^{z-\beta}f(x)+\tilde{S}_{j,\beta}^{R,t}g(x)B^{\psi_0}_{R}B_{Rt}^{z-\beta}f(x)]t^{2(z-\beta)+1}dt,
\end{eqnarray*}
where $c_z=\frac{\Gamma(z+1)}{\Gamma(\beta)\Gamma(z-\beta+1)}$.

Therefore, we see that the derivative is given by 
\begin{eqnarray*}
    \partial_z \left(\tilde{\g}^{z}_{j,R}(f,g)(x)\right)
    &=& \left(\partial_zc_{z}\right)\int_0^{\sqrt{2^{-j+1}}}[S_{j,\beta}^{R,t}g(x)B^{\psi_0}_{R}A_{Rt}^{z-\beta}f(x)+\tilde{S}_{j,\beta}^{R,t}g(x)B^{\psi_0}_{R}B_{Rt}^{z-\beta}f(x)]t^{2(z-\beta)+1}dt\\
    &&~+ c_{z}\int_0^{\sqrt{2^{-j+1}}}[S_{j,\beta}^{R,t}g(x)B^{\psi_0}_{R}\tilde{A}_{Rt}^{z-\beta}f(x)+\tilde{S}_{j,\beta}^{R,t}g(x)B^{\psi_0}_{R}\tilde{B}_{Rt}^{z-\beta}f(x)]t^{2(z-\beta)+1}dt\\
    &&~+ c_{z}\int_0^{\sqrt{2^{-j+1}}}[S_{j,\beta}^{R,t}g(x)B^{\psi_0}_{R}A_{Rt}^{z-\beta}f(x)+\tilde{S}_{j,\beta}^{R,t}g(x)B^{\psi_0}_{R}B_{Rt}^{z-\beta}f(x)]t^{2(z-\beta)+1}\log tdt\\
    &=& \tilde{I}_{j,R}+\tilde{II}_{j,R}+\tilde{III}_{j,R}
\end{eqnarray*}
The $L^2\times L^2\to L^1$ estimates for the square functions corresponding to the expressions $\tilde{I}_{j,R}$ and $\tilde{III}_{j,R}$ can be obtained similarly as for the operator $\G_{j}^z$. We will need to make use of $L^2$-boundedness of $f\rightarrow \left(\sup_{R>0}R_j^{-1}\int_0^{R_j}|B^{\psi_0}_{R}B_t^{\delta}f(x)|^2 dt\right)^{1/2}$ from [\cite{JS}, page 24] and the operator $f\rightarrow \left(\int_0^{\sqrt{2^{-j+1}}}\int_0^\infty|B^{\psi_0}_{R}A_{Rt}^{\delta}f(x)|^2 dt\right)^{1/2}$ from [\cite{CKSS}, page 15]. 

Finally, the square function for the second expression $\tilde{II}_{j,R}$ is dealt with as follows. 

\begin{eqnarray*} 
    &&	\left\|\left(\int_0^\infty|\tilde{II}_{j,R}|^2\frac{dR}{R}\right)^{\frac{1}{2}}\right\|_1\\
	&\lesssim& |c_{z}|\left\|\sup_{R>0}\left(\int_{0}^{\sqrt{2^{-j+1}}}|S_{j,\beta}^{R,t}g(\cdot)t^{2\delta+1}|^2dt\right)^{\frac{1}{2}}\right\|_2  \left\|\left[\int_0^\infty\left(\int_{0}^{\sqrt{2^{-j+1}}}|B^{\psi_0}_{R}\tilde{A}_{Rt}^{\delta}f(x)|^2dt\right)\frac{dR}{R}\right]^{\frac{1}{2}}\right\|_2\\
	&&~+ 2^{-\frac{j}{4}}|c_{z}|\left\|\left(\int_0^\infty\int_{0}^{\sqrt{2^{-j+1}}}|\tilde{S}_{j,\beta}^{R,t}g(\cdot)t^{2\delta+1}|^2dt\frac{dR}{R}\right)^{\frac{1}{2}}\right\|_2 \left\|\sup_{R>0}\left(\frac{1}{R_j}\int_0^{R_j}|B^{\psi_0}_{R}\tilde{B}_t^{\delta}f(x)|^2dt\right)^{\frac{1}{2}}\right\|_2
\end{eqnarray*}
Recall that the operator $f\rightarrow \left(\sup_{R>0}R_j^{-1}\int_0^{R_j}|B^{\psi_0}_{R}\tilde{B}_t^{\delta}f(x)|^2 dt\right)^{1/2}$ satisfies the desired $L^2$ estimates, see  equation~\eqref{L21}.  Also, the $L^2$ estimates for  $f\rightarrow \sup_{R>0}\left(\int_{0}^{\sqrt{2^{-j+1}}}|S_{j,\beta}^{R,t}g(\cdot)t^{2\delta+1}|^2dt\right)^{\frac{1}{2}}$ and $f\rightarrow \left(\int_0^\infty\int_{0}^{\sqrt{2^{-j+1}}}|\tilde{S}_{j,\beta}^{R,t}g(\cdot)t^{2\delta+1}|^2dt\frac{dR}{R}\right)^{\frac{1}{2}}$ are known from [\cite{JS}, Theorem 5.1] and [\cite{CKSS}, Theorem 3.2] respectively. Therefore,  in order to conclude the $L^2\times L^2\to L^1$ boundedness of the square function we only need to prove the following estimate. 
\begin{eqnarray}\label{L22}
	\left\|\left(\int_0^\infty\int_{0}^{\sqrt{2^{-j+1}}}|B^{\psi_0}_{R}\tilde{A}_{Rt}^{\delta}f(\cdot)|^2dt\frac{dR}{R}\right)^{\frac{1}{2}}\right\|_2\lesssim 2^{\frac{5}{4}(-j+1)}\|f\|_2~\text{for}~ \text{Re}(\delta)>-\frac{1}{2}\end{eqnarray}

	Consider 
\begin{eqnarray*}
    \left(\int_0^\infty\int_{0}^{\sqrt{2^{-j+1}}}|B^{\psi_0}_{R}\tilde{A}_{Rt}^{\delta}f(x)|^2dt\frac{dR}{R}\right)^{\frac{1}{2}}
    &=&\left(\int_{0}^{\sqrt{2^{-j+1}}}\int_0^\infty|B^{\psi_0}_{R}\tilde{A}_{Rt}^{\delta}f(x)|^2\frac{dR}{R}dt\right)^{\frac{1}{2}}\\
	&\lesssim& 2^{\frac{5}{4}(-j+1)}\left(\int_0^\infty|B^{\psi_0}_{R}\tilde{A}_{Rt}^{\delta}f(x)|^2\frac{dR}{R}\right)^{\frac{1}{2}}\\
&\lesssim& 2^{\frac{5}{4}(-j+1)}	\left(\int_0^{\infty}|M(\tilde{A}_t^{\delta}f(x)|^2 t^{-1}dt\right)^{1/2}
\end{eqnarray*}
Here we have used that $\sup_{R>0}|B^{\psi_0}_{R}f(x)|\lesssim  M(f)(x)$. Finally,  
invoking vector-valued estimates for the Hardy-Littlewood Maximal function, see~\cite{DK, FS}, and $L^2$ estimate for the square function corresponding to $\tilde{A}_{t}^{\delta}$ we get the desired estimate~\eqref{L22}. 
This completes the proof of the key Lemma~\ref{keylem2}.  
\qed 
\section*{Acknowledgement}  The first author is supported by CSIR (NET), file no. 09/1020(0182)/2019-EMR-I. The second author acknowledges the support from Science and Engineering Research Board, Department of Science and  Technology, Govt. of India under the scheme Core Research Grant with file no. CRG/2021/000230.

\end{document}